\pdfoutput=1

\newif\ifjournalstyle

\ifjournalstyle

  \documentclass[12pt, twoside]{article}
  \usepackage{amsmath,amsthm,amssymb}
  \usepackage{times}
  \usepackage{enumerate}

  \theoremstyle{definition}

  \numberwithin{equation}{section}

\else
  \documentclass[a4paper,oneside]{amsart}
  \pdfoutput=1
\fi

\usepackage{graphicx}
\usepackage{subcaption}
\usepackage{url}
\usepackage{hyperref}
\usepackage[capitalize]{cleveref}
\renewcommand{\PackageWarningNoLine}[2]{}
\usepackage{amsmath}
\usepackage{amssymb}
\usepackage{amsthm}
\usepackage{amsfonts}
\usepackage{mathrsfs}
\usepackage{esint}
\usepackage{colonequals}
\usepackage{color}
\usepackage{tikz}
\usepackage{stmaryrd}

\graphicspath{{gfx/}}

\newcommand{\eg}{e.\,g.\ }
\newcommand{\ie}{i.\,e.\ }

\newcommand{\N}{\mathbb{N}}
\newcommand{\R}{\mathbb{R}}

\newcommand{\cA}{\mathcal A}
\newcommand{\cB}{\mathcal B}

\newcommand{\cF}{\mathcal F}
\newcommand{\cG}{\mathcal G}

\newcommand{\cJ}{\mathcal J}

\newcommand{\cP}{\mathcal P}

\newcommand{\cV}{\mathcal V}

\newcommand{\cX}{\mathcal X}

\newcommand{\cZ}{\mathcal Z}

\newcommand{\abs}[1]{ \left\vert #1 \right\vert }
\newcommand{\mat}[1]{\begin{pmatrix} #1 \end{pmatrix}}
\newcommand{\norm}[1]{ \left\Vert #1 \right\Vert }

\DeclareMathOperator*{\argmin}{arg\,min}

\renewcommand{\div}{\operatorname{div}}
\DeclareMathOperator*{\Tr}{Tr}

\renewcommand{\d}{\,\mathrm{d}}
\newcommand{\dd}[1]{ \frac{\partial}{\partial #1} }
\newcommand{\eps}{\varepsilon}

\newcommand{\wt}[1]{ \widetilde{#1} }

\newtheorem{theorem}{Theorem}[section]
\newtheorem{lemma}[theorem]{Lemma}

\numberwithin{equation}{section}
\numberwithin{figure}{section}

\usepackage{xifthen}

\newcommand{\p}{\textbf{p}}
\newcommand{\q}{\textbf{q}}
\renewcommand{\u}{\textbf{u}}
\newcommand{\e}{\textbf{e}}

\begin{document}




\ifjournalstyle

  \title{On differentiability of the membrane-mediated mechanical interaction energy of discrete--continuum membrane--particle models}

  \author{%
    Carsten Gr\"aser\\
    Freie Universit\"at Berlin, Institut f\"ur Mathematik\\
    Arnimallee~6, D-14195~Berlin, Germany\\
    graeser@mi.fu-berlin.de\\
    \\
    Tobias Kies\\
    Freie Universit\"at Berlin, Institut f\"ur Mathematik\\
    Arnimallee~6, D-14195~Berlin, Germany\\
    tobias.kies@fu-berlin.de
  }

\else

  \title[Differentiability of membrane-mediated interaction energy]{On differentiability of the membrane-mediated mechanical interaction energy of discrete--continuum membrane--particle models}

  \author[Gr\"aser]{Carsten Gr\"aser}
  \address{Carsten Gr\"aser\\
  Freie Universit\"at Berlin\\
  Institut f\"ur Mathematik\\
  Arnimallee~6\\
  D-14195~Berlin\\
  Germany}
  \email{graeser@mi.fu-berlin.de}

  \author[Kies]{Tobias Kies}
  \address{Tobias Kies\\
  Freie Universit\"at Berlin\\
  Institut f\"ur Mathematik\\
  Arnimallee~6\\
  {D-14195}~Berlin\\
  Germany}
  \email{tobias.kies@fu-berlin.de}

\fi

\maketitle

\subsection*{Acknowledgement}
This research has been partially funded by Deutsche Forschungsgemeinschaft (DFG)
through the grant CRC 1114: “Scaling Cascades in Complex Systems”,
Project Number 235221301,
Project AP01 "Particles in lipid bilayers"
and A07 "Langevin dynamics of particles in membranes".

\begin{abstract}
  We consider a discrete--continuum model of a biomembrane
  with embedded particles. While the membrane is represented by
  a continuous surface, embedded particles are described by
  rigid discrete objects which are free to move and rotate in
  lateral direction. For the membrane we consider a linearized
  Canham--Helfrich energy functional and height and slope
  boundary conditions imposed on the particle boundaries
  resulting in a coupled minimization problem for the membrane
  shape and particle positions.

  When considering the energetically optimal membrane shape
  for each particle position we obtain a reduced energy functional
  that models the implicitly given interaction potential for the
  membrane-mediated mechanical particle--particle interactions.
  We show that this interaction potential is differentiable with
  respect to the particle positions and orientations.
  Furthermore we derive a fully practical representation
  of the derivative only in terms of well defined derivatives
  of the membrane.
  This opens the door for the application of minimization
  algorithms for the computation of minimizers of the coupled
  system and for further investigation of the interaction potential
  of membrane-mediated mechanical particle--particle interaction.

  The results are illustrated with numerical examples comparing
  the explicit derivative formula with difference quotient approximations.
  We furthermore demonstrate the application of the derived formula to implement a gradient flow
  for the approximation of optimal particle configurations.
\end{abstract}

\section{Introduction}

  Particles embedded into membranes are commonly expected to be crucial for
  various biological processes involving the shaping of the membrane.
  Examples of such particles are transmembrane and bar-domain proteins.
  For example the latter are conjectured to play an important role
  in early stages of clathrin-mediated endocytosis \cite{HauckeKozlov2018}.
  The reason for the importance of proteins for membrane shaping is,
  that they may induce local membrane deformations in their vicinity.
  Since the membrane itself consists of a lipid bilayer,
  which---in the lateral direction---can be seen as a fluid,
  such particles are able to move easily within the membrane.
  As a consequence the local particle-induced membrane deformation
  implicitly induces a membrane-mediated mechanical particle--particle interaction.
  Driven by the corresponding interaction potential, particles may cluster
  and form energetically preferable patterns.

  This observation and the fact that membrane-shaping particles seem
  to be crucial for biological membrane functions stimulated various
  research directions on such particles and membrane-mediated particle--particle
  interactions.
  A common approach is based on atomistic or coarse-grained molecular dynamics
  (MD) models of the membrane.
  To overcome the severe length scale limitations of such approaches,
  alternative modeling techniques representing the membrane as
  continuous surface that minimizes an elastic energy have been
  developed \cite{Canham1970,Helfrich1973}.
  Using such an models it was shown that there are long-range
  interactions between particles that are predominantly membrane-mediated
  \cite{GoBrPi93}. Since then further work has been done to investigate
  particle interactions within elasticity models.  Typically a flatness
  assumption on the membrane is made and particles are modeled as circular
  disks or points and their coupling to the membrane is prescribed by radially
  symmetric boundary conditions for which the interaction energy can either be
  computed analytically or approximately by asymptotic expansion
  \cite{KiNeOs98,WeKoHe98,DoFo02,YoHaDe14,FoGa15}.  However, it turns out that
  the shape of particles has a significant impact on their interaction
  \cite{KiNeOs00}.  More recent work is also interested in numerical
  computations with pattern formation of many non-circular particles
  \cite{HaPh13,KaKoKlHa16}, and attention was also given to situations where the
  flatness assumption is no longer fulfilled \cite{ElliottFritzHobbs2017,ReDe11,ScKo15}.
  Also more elaborate models for proteins in continuum elastic models have recently been
  considered in \cite{ArBeMaGr16}.

  To understand the pattern formation of particles it is desirable to
  quantify the forces exerted on the particles by the membrane in a framework
  that is as widely applicable as possible.
  General results in this direction have been obtained based on arguments
  from differential geometry \cite{Deserno2015}.
  The methods derived therein give insight into the
  qualitative behavior of particle interactions, but---to the best of the
  authors' knowledge---they have not yet been made fully available for
  numerical computations.

  In this paper we consider a discrete--continuum model where the membrane is
  modeled as a continuous graph minimizing a linearized Canham--Helfrich
  bending energy and where an arbitrary amount of particles are embedded
  into the membrane.
  These particles are
  modeled as discrete entities which are coupled to the membrane through
  certain boundary conditions.  As particles are free to move in the membrane,
  those boundary conditions depend on each particle's position.  Consequently,
  the overall system's energy given fixed boundary conditions and an optimal
  membrane shape can be written as a function of the particle positions, which
  we call the \emph{interaction energy}.

  In this setting we propose a method to prove differentiability of the
  interaction energy for arbitrary shapes and boundary conditions.
  Furthermore, we derive an expression for the derivative that can be evaluated
  numerically within a finite element scheme and where the evaluation error is
  bounded in terms of the discretization error of the finite element
  approximation.  The proof is based on an application of the implicit function
  theorem and ideas from shape calculus \cite{HiGr27,SoZo11}.  As such, the
  method is rather general and hence it naturally extends to a wider class of
  models that for example use nonlinear elastic energies or certain other
  membrane--particle couplings.

  In the following we give an outline of this paper.
  In \cref{sec:model} we introduce the Canham--Helfrich energy
  in Monge-gauge as a
  model for the membrane and parametric boundary conditions for the coupling of
  the particles.  \cref{sec:interaction} is then concerned with further
  mathematical notation that we use in order to define the interaction energy.
  There we also reformulate the parametric boundary conditions as linear
  constraints by using trace operators and appropriate projection operators.
  Afterwards, in \cref{sec:differentiation}, we prove differentiability of the
  interaction energy and derive a numerically feasible expression for the
  gradient.  Finally, \cref{sec:examples} shows some example computations that
  illustrate that the derived formula can indeed be applied in a numerical
  scheme.

\section{Membrane and particle model}\label{sec:model}

  The Canham--Helfrich model of a membrane given as two-dimensional surface
  $\mathcal{M} \subset \R^3$ is based on the elastic
  bending energy
  \begin{align*}
    J_{\text{CH}}(\mathcal{M}) \colonequals \int_{\mathcal{M}} \frac{\kappa}{2} (H-c_0)^2 + \kappa_G K + \sigma \d{\mathcal{M}}
  \end{align*}
  where $H$ and $K$ denote the mean and Gaussian curvature
  with corresponding bending rigidities $\kappa$ and $\kappa_G$,
  $\sigma$ is the surface tension, and
  $c_0$ is the so called spontaneous curvature.
  Considering membranes with fixed topology and fixed geodesic curvature
  the Gaussian curvature term is a constant and can be dropped.
  For simplicity we will also restrict our considerations to the
  case of vanishing spontaneous curvature $c_0=0$.

  Assuming that the membrane is given by a graph $\mathcal{M}=\{(x,u(x)) \,|\, x \in \Omega\}$
  which is almost flat in the sense that $|\nabla u| \ll 1$ we consider
  the well-established \emph{linearized Canham--Helfrich model in Monge-gauge:}
  Given a 2-dimensional reference domain $\Omega \subseteq \R^2$ and a function
  $u \in H^2(\Omega)$, the membrane shape is described by the graph
  of $u$.
  The bending energy of this membrane is approximated by
  \begin{align*}
    J(\Omega, u) \colonequals \frac{1}{2} \int_{\Omega} \kappa (\Delta u(x))^2 + \sigma \norm{\nabla u(x)}^2 \d{x}
  \end{align*}
  where $\kappa>0$ and $\sigma \geq 0$ denote the
  \emph{bending rigidity} and the \emph{surface tension}, respectively. 
  It is
  noted that this corresponds to the geometric linearization of the full nonlinear
  Canham--Helfrich energy $J_{\text{CH}}$ near a flat membrane with $\nabla u = 0$.
  For more details on the model we refer to \cite{ElGrHoKoWo16} and the references therein.

  In absence of particles interacting with the membrane, this model
  determines the stationary shape of the membrane solely by minimizing the
  energy $J$. In the following we explain how the embedded particles
  are coupled to the membrane,
  before we state the model problem that is
  central to this paper.

  For simplicity we first consider a single \emph{transmembrane protein} that interacts
  with the membrane.
  Such a protein is not merely connected to one side of the membrane surface
  but rather is included in the membrane.
  The reason for this situation is that the protein has a \emph{hydrophobic belt}
  which is shielded from the surrounding water by the membrane lipids.
  As a consequence the membrane preferably connects to the belt which
  in turn deforms the membrane according to its shape.
  We assume that the particle does not undergo any deformation
  and denote its rigid 3-dimensional shape $\cB \subseteq \R^3$.
  
  In the following we will suppose that the hydrophobic belt
  is approximated by a curve $\cG$.  We assume that $\cG$ is a simple
  closed curve that can be parameterized over the 2-dimensional Euclidean
  plane.  This means that there exists a simple closed curve $\Gamma = \partial B \subseteq
  \R^2$ and a continuous function $g_{0}\colon \Gamma \rightarrow \R$ such that
  $\cG = \{ (x,g_{0}(x)) \mid x \in \Gamma\}$.  This gives rise to the boundary
  condition
  $u|_{\Gamma} = g_0\text{, }$
  which models that the membrane is connected to the particle at the interface $\cG$.
  The fact that the membrane is connected perpendicular to the hydrophobic
  region is modeled by additionally assuming that it is attached to $\cG$
  with a fixed slope.
  To this end, we impose the additional boundary condition $\partial_\nu u|_{\Gamma} = g_1$
  with a function $g_1\colon \Gamma \rightarrow \R$ describing the slope.
  Here $\nu$ is an oriented unit normal on $\Gamma$ and $\partial_\nu
  u|_{\Gamma}$ denotes the normal derivative of $u$ on $\Gamma$.
  The situation is illustrated in Figure~\ref{fig:transmembrane}
  showing the hydrophobic belt $\cG$, its projection to the plane $\Gamma$,
  and the boundary conditions on $\Gamma$.

  \begin{figure}[ht]
    \begin{center}
      \input{gfx/particle_side_view.pgf}%
      \caption{%
        Side view of a particle with hydrophobic belt $\cG$, its projection to the plane $\Gamma$,
        and the height contour $g_0$ and slope $g_1$ prescribed at $\Gamma$.
        }%
        \label{fig:transmembrane}
    \end{center}
  \end{figure}

  Those constraints do not yet account for the fact that the particle is in
  principle free to move in space.  To this end we parameterize the current
  position of the particle using translations $z_j$ along the $x_j$-axes and
  rotations $\alpha_j$ around the $x_j$-axes.  More precisely, let $\cB^0$ be a
  reference state of the particle that is centered in the origin
  with hydrophobic belt $\cG^0 = \{ (\hat{x},g_{0}^0(\hat{x})) \mid \hat{x} \in \Gamma^0\}$ for some $\Gamma^0$
  and let
  $R_j(\alpha_j) \in \R^{3\times 3}$ be the $\alpha_j$-rotation matrix around
  the $x_j$-axis.
  Then for $\hat{y} \in \cB^0$ the rigid body transformation $\Psi(z,\alpha): \cB^0 \to \R^3$
  corresponding to $(z,\alpha)$ is given by
  \begin{align*}
    \Psi(z,\alpha)(\hat{y}) \colonequals \Psi(z,\alpha;\hat{y}) \colonequals R_1(\alpha_1)R_2(\alpha_2)R_3(\alpha_3)\hat{y} + z
  \end{align*}
  and we define the parameterized particle as
  \begin{align*}
    \cB = \cB(z,\alpha)
    \colonequals \left\{\Psi(z,\alpha)(\hat{y}) \mid \hat{y} \in \cB^0 \right\}.
  \end{align*}
  For the hydrophobic belt straight forward application of the parameterization
  would lead to
  \begin{align*}
    \cG = \cG(z,\alpha)
    \colonequals \left\{\Psi(z,\alpha)(\hat{y}) \mid \hat{y} \in \cG^0 \right\}.
  \end{align*}
  and the corresponding projection $\Gamma = \{(y_1,y_2) \mid y \in \cG\}$
  to the plane.
  Given the particle slope $g_1^0: \Gamma^0 \to \R$ on the reference curve,
  the membrane--particle constraints introduced above could be imposed
  to the parametrized particle by first transforming the membrane given by the graph of $u$
  according to $\Psi(z,\alpha)^{-1}$ and then imposing constraints
  \begin{align*}
    \hat{u}|_{\Gamma^*} &= g_0^0,&
    \partial_{\hat{\nu}} \hat{u}|_{\Gamma^0} &= g_1^0
  \end{align*}
  on the transformed graph of (if existing) $\hat{u}$ over the reference curve $\Gamma^0$
  with unit normal $\hat{\nu}$.
  However, this approach suffers from several drawbacks related to the
  rotations $R_1(\alpha_1)$ and $R_2(\alpha_2)$.
  First, they can easily exceed the regime where $\cG$ can be written as a graph over $\Gamma$
  and where the transformation of the graph of $u$ is a graph of some $\hat{u}$.
  Second, the shape of the projected curve $\Gamma$ differs from $\Gamma^0$
  and normals of $\Gamma^0$ do not transform to normals of $\Gamma$.
  To avoid these complications, we will simplify the transformation
  assuming small rotational angles:

  On the one hand we make the assumption that the
  reference set $\cB^0$ is oriented in such a way that the belt $\cG^0$ is
  'almost flat', by which we mean that $\max_{\hat{x} \in \Gamma^0} \abs{g^0_{0}(\hat{x})} /
  \norm{\hat{x}}$ is small.
  On the other hand we assume that the angles $\alpha_1$ and $\alpha_2$
  are small, such that the parameterized belt is still 'almost flat'.
  Based on these assumptions we first replace $R_1(\alpha_1)$ and $R_2(\alpha_2)$
  by first order Taylor expensions  $\tilde{R}_1(\alpha_1)$ and $\tilde{R}_2(\alpha_2)$
  near $\alpha_1=0=\alpha_2$
  leading to
  \begin{align*}
    R_1(\alpha_1)R_2(\alpha_2) y
    \approx
    \tilde{R}_1(\alpha_1) \tilde{R}_2(\alpha_2) y
    =
    \begin{pmatrix}
      1 & 0 & -\alpha_2\\
      -\alpha_1\alpha_2 & 1 & -\alpha_1\\
      \alpha_2 & \alpha_1 & 1
    \end{pmatrix}
    y
    =
    \begin{pmatrix}
      y_1 -\alpha_2y_3\\
      y_2 -\alpha_1\alpha_2y_1 -\alpha_1 y_3\\
      y_3 + \alpha_2y_1 + \alpha_1 y_2
    \end{pmatrix}.
  \end{align*}
  Using $y= R_3(\alpha_3)\hat{y}$
  for $\hat{y} = (\hat{x}, g_0^0(\hat{x})) \in \cG^0$
  and dropping higher order terms in
  $\alpha_1$, $\alpha_2$,
  and $y_3=\hat{y}_3 = g_0^0(\hat{x})$
  finally leads to the approximate transformation
  \begin{align*}
    \wt{\Psi}(z,\alpha)(\hat{y})
    \colonequals
    \begin{pmatrix}
      1 & 0 & 0\\
      0 & 1 & 0\\
      \alpha_2 & \alpha_1 & 1
    \end{pmatrix}
    R_3(\alpha_3) \hat{y} + z
    \approx
    \Psi(z,\alpha)(\hat{y})
  \end{align*}
  and the corresponding parametrized particle belt
  \begin{align*}
    \cG(z,\alpha)
    \colonequals \left\{\wt{\Psi}(z,\alpha)(\hat{y}) \mid \hat{y} \in \cG^0 \right\}.
  \end{align*}

  In the following we use the shortcut $p = (z_1,z_2,\alpha_3)$
  for the in-plane translation and rotation.
  Introducing
  \begin{align}
    \label{eq:rotationMatrix}
    R(\alpha) \colonequals \mat{ \cos(\alpha) & -\sin(\alpha) \\ \sin(\alpha) & \cos(\alpha) }\text{.}
  \end{align}
  the in-plane transformation induced by $p$ is given by the rigid body motion
  \begin{align}
    \label{eq:transRotation}
    \varphi(p)(\hat{x}) \colonequals \varphi(p;\hat{x})  & \colonequals R(\alpha_3)\hat{x} + \mat{z_1\\z_2}
  \end{align}
  with its inverse
  \begin{align*}
    \varphi^{-1}(p;x) \colonequals \varphi(p)^{-1}(x) \colonequals R(-\alpha_3)\left( x - \mat{z_1\\z_2}\right).
  \end{align*}
  Utilizing $\varphi$ we can write $\wt{\Psi}$ on $\cG^0$ as
  \begin{align*}
    \wt{\Psi}(z,\alpha)(\hat{x},g_0^0(\hat{x}))
    =
    \begin{pmatrix}
      \varphi(p)(\hat{x}) \\
      g_0^0(\hat{x}) + z_3 + (\alpha_2, \alpha_1)R(\alpha_3) (\hat{x})
    \end{pmatrix}.
  \end{align*}
  Thus $\cG(z,\alpha)$ can be written as a graph over its projection to
  the plane
  \begin{align*}
    \Gamma(p)
      \colonequals \Gamma(z,\alpha)
      \colonequals \left\{(y_1,y_2) \mid y \in \cG(z,\alpha) \right\}
      = \left\{\varphi(p)\hat{x} \mid \hat{x} \in \Gamma^0\right\}
      = \varphi(p;\Gamma^0).
  \end{align*}
  Figure~\ref{fig:translated_particle} illustrates the
  reference particle curve $\Gamma^0$ and the transformed particle curve
  $\Gamma(p) = \Gamma(z_1,z_2,\alpha_3)$.
  The figure shows a top view that complements the side view
  of Figure~\ref{fig:transmembrane}.

  \begin{figure}[ht]
    \begin{center}
      \input{gfx/particle_top_view.pgf}%
      \caption{%
        Top view of reference particle $\Gamma^0$ and moved particle $\Gamma(z_1,z_2,\alpha_3)$.
        }%
        \label{fig:translated_particle}
    \end{center}
  \end{figure}

  The zeroth and first order boundary condition for the parameterized
  particle are now obtained by pulling back the graph of $u$ using
  the approximate transformation $\wt{\Psi}(z,\alpha)^{-1}$
  and then imposing the conditions for $g^0_0$ and $g^0_1$ on $\Gamma^0$.
  Noting that $\Gamma^0 = \varphi^{-1}(p;\Gamma(p))$
  and defining $g_k(p;x) \colonequals g_k^0(\varphi^{-1}(p;x))$
  the boundary conditions then read
  \begin{align*}
    u(x) & = g_0(p;x) + \alpha_2 (x_1-z_1) + \alpha_1(x_2-z_2) + z_3 && \text{on $\Gamma(z_1,z_2,\alpha_3)$}, \\
    \partial_\nu u(x) & = g_1(p;x) + \alpha_2 \nu_1 + \alpha_1 \nu_2 && \text{on $\Gamma(z_1,z_2,\alpha_3)$},
  \end{align*}
  where $\nu_1$ and $\nu_2$ denote the components of
  an oriented unit normal on $\Gamma(p)$.

  The dependence of these boundary conditions
  on the parameters $(z,\alpha)$ can be split naturally into the
  nonlinear dependence on the in-plane position $p=(z_1,z_2,\alpha_3)$
  and the linear dependence on variations of the height
  and tilt given by $(z_3,\alpha_1,\alpha_2)$.
  We will exploit the latter by minimizing the membrane energy with
  respect to $u$ and $(z_3,\alpha_1,\alpha_2)$ simultaneously.
  It can easily be seen, that this can be written equivalently be factoring
  out $(z_3,\alpha_1,\alpha_2)$ from the boundary condition in the
  sense that they are only enforced up to an arbitrary selection of $(z_3,\alpha_1,\alpha_2)$.
  Thus the boundary conditions become \emph{parametric boundary
  conditions}:
  \begin{align}
    \exists \gamma \in \R^3\colon
    \Biggl\{
    \begin{split}
      \label{eq:parametricBC}
      u|_{\Gamma(p)}(x) & = g_0(p;x) + \gamma_1 x_1 + \gamma_2 x_2 + \gamma_3
      \\ \partial_\nu u|_{\Gamma(p)}(x) & = g_1(p;x) + \gamma_1 \nu_1 + \gamma_2 \nu_2\text{.}
    \end{split}
    \qquad\forall x \in \Gamma(p).
  \end{align}
  Notice that, $z_3$, $\alpha_1$, and $\alpha_2$ are no longer relevant for
  describing the particle's position and are now rather implicit to the
  boundary conditions.  A particle in the model is then solely determined by
  its reference curve $\Gamma^0$, its reference boundary conditions $g_0^0,
  g_1^0$, and its position $p = (z_1, z_2, \alpha_3)$ in the Euclidean plane.

  Parameterized boundary conditions for particle--membrane coupling
  with variable height and tilt have first been considered in~\cite{ElGrHoKoWo16}.
  They can be interpreted physically as the particles being tied
  only to the membrane such that they can freely change their height
  and tilt angle with the membrane. In contrast, the lateral motion
  of particles is an independent process. This is motivated by the fact,
  that the membrane has a bending rigidity in normal direction but behaves
  like a viscous fluid in tangential direction.

  In the case where multiple particles are present we state the constraints
  analogously by imposing the above constraints for each particle separately.

\section{Interaction energy}\label{sec:interaction}

  Before we can formulate the final model problem, we need to introduce some
  notation.  We also augment the parametric boundary conditions by Dirichlet
  boundary conditions on the outer boundary $\partial\Omega$ which reflects the
  fact that we consider a small almost flat patch of the full membrane.
  For a discussion of other possible boundary conditions we refer to~\cite{ElGrHoKoWo16}.
  Particle--membrane coupling and outer boundary conditions
  will both be formulated using simple linear operators in the following.

  We consider $N$ particles with reference curves $\Gamma_i^0$, height profiles
  $g_{i0}^0$ and slopes $g_{i1}^0$.  Given a \emph{particle configuration}
  $\p = (\p_i)_{i=1,\dots,N} \in \R^{N\times 3}$ we define the curves
  \begin{align*}
    \Gamma_i(\p_i)  & \colonequals \{ \varphi(\p_i;y) \mid y \in \Gamma_i^0 \}
  \end{align*}
  and $\Gamma_0 \colonequals \Gamma_0^0 \colonequals \partial\Omega$.
  Their union is $\Gamma(\p) \colonequals \bigcup_{i=0}^N \Gamma_i(\p_i)$ where we use
  $\p_0 \colonequals 0$ and $\Gamma_0(\p_0) \colonequals \Gamma_0$ for the sake of a consistent
  notation.
  For $i > 0$ we denote the set enclosed by $\Gamma_i(\p_i)$ by $B_i(\p_i)$ and the
  union of these is denoted by $B(\p) \colonequals \bigcup_{i=1}^N B_i(\p_i)$.
  We define the $\p$-dependent reference domain as $\Omega(\p) \colonequals \Omega
  \setminus B(\p)$.
  
  Based on this we define the interior of the set of \emph{feasible particle
  configurations} as
  \begin{align*}
    \Lambda^\circ \colonequals \left\{ \p \in \R^{N\times 3} \mid \forall{i,j \in \{1,\dots,N\}, i \neq j\colon}\ \Gamma_i(\p_i) \subseteq \Omega^\circ \ \wedge \ \overline{B_i(\p_i)} \cap \overline{B_j(\p_j)} = \emptyset\right\}\text{,}
  \end{align*}
  and set $\Lambda \colonequals \overline{\Lambda^\circ}$.
  
  Now suppose $\p \in \Lambda$.
  Then we define the trace operators
  \begin{align}
    \label{eq:traceOperator}
    \begin{aligned}
      T_i(\p)\colon H^2(\Omega(\p)) & \longrightarrow H^{3/2}(\Gamma_i^0) \times H^{1/2}(\Gamma_i^0)
      \\ u & \longmapsto \mat{ u|_{\Gamma_i(\p_i)} \circ \varphi(\p_i) \\ (\partial_\nu u|_{\Gamma_i(\p_i)}) \circ \varphi(\p_i) }\text{.}
    \end{aligned}
  \end{align}
  Here $\nu$ is the unit outer normal on $\Gamma(\p)$ with respect to the domain $\Omega(\p)$.
  We define the joint trace operator by $T(\p)u \colonequals (T_i(\p)u)_{i=0,\dots,N}$.
  We also define $g_{0k}^0 \colonequals 0$, $g \colonequals ( (g_{i0}^0, g_{i1}^0) )_{i=0,\dots,N}$,
  and $g_i = (g_{i0}^0, g_{i1}^0)$.
  
  In the next step we prove a useful reformulation of the parametric boundary
  conditions as linear constraints.
  The aim is to get rid of the parameters $\gamma$ in conditions of the
  form~\eqref{eq:parametricBC} for the $i$-th particle
  by writing it as $P_iT_i(\p)u = P_ig_i$ for a suitable projection operator
  \begin{align*}
    P_i\colon L^2(\Gamma_i^0) \times L^2(\Gamma_i^0) & \longrightarrow L^2(\Gamma_i^0) \times L^2(\Gamma_i^0).
  \end{align*} 
  For the domain boundary represented by $i=0$ we can simply use
  $P_0(v_1,v_2) \colonequals (v_1,v_2)$.
  While all $P_i$ are different, the construction works exactly
  the same for all $i \in \{1,\dots,N\}$.
  Thus, to simplify the notation,
  we will drop the index $i$ in the following construction of $P_i = P$.
  Define for $\hat{x} \in \R^2$ the functions
  \begin{align*}
      \eta_{1}(\hat{x}) \colonequals \hat{x}_1,
    \qquad \eta_{2}(\hat{x}) \colonequals \hat{x}_2,
    \qquad \eta_{3}(\hat{x}) \colonequals 1
    \text{,}
  \end{align*}
  spanning the space $V\colonequals \operatorname{span}\{\eta_1, \eta_2, \eta_3\}$.

  We will define $P$ such that
  $\{(v, \partial_{\hat{\nu}} v) \mid v \in V\} \subset L^2(\Gamma^0) \times L^2(\Gamma^0)$
  is the kernel of $P$ where $\hat{\nu}$ again denotes the unit normal to $\Gamma^0$.
  To this end let
  $\cP: L^2(\Gamma^0) \to V$ be the $L^2(\Gamma^0)$-orthogonal projection into $V$.
  Notice that $\cP$ can easily be computed using
  \begin{align*}
    \cP(v) = \Theta \left( G^{-1}
    \begin{pmatrix}
      \langle v, \eta_1 \rangle_{L^2(\Gamma^0)} \\
      \langle v, \eta_2 \rangle_{L^2(\Gamma^0)} \\
      \langle v, \eta_3 \rangle_{L^2(\Gamma^0)}
    \end{pmatrix}
    \right)
  \end{align*}
  where $G \in \R^{3 \times 3}$ is the Gramian matrix with
  $G_{jk} = \langle \eta_k, \eta_j \rangle_{L^2(\Gamma^0)}$
  and $\Theta: \R^3 \to V$ the coordinate isomorphism
  with $\Theta(\xi) = \xi_1 \eta_1 + \xi_2 \eta_2 + \xi_3 \eta_3$.
  Now we define
  $P \colon L^2(\Gamma^0) \times L^2(\Gamma^0) \longrightarrow L^2(\Gamma^0) \times L^2(\Gamma^0)$ by
  \begin{align*}
    P(v_1, v_2) =
    \begin{pmatrix}
      v_1 - \cP(v_1)|_{\Gamma^0} \\
      v_2 - \partial_{\hat{\nu}} \cP(v_1)|_{\Gamma^0}
    \end{pmatrix}.
  \end{align*}
  Since $\cP$ is an orthogonal projection into $V$, we have
  \begin{align*}
    0 = P(v_1,v_2)
    \qquad\Longleftrightarrow\qquad
    v_1 \in V \text{ and } v_2 = \partial_{\hat{\nu}} v_1.
  \end{align*}
  Hence $\{(v, \partial_{\hat{\nu}}  v) \mid v \in V\} \subset L^2(\Gamma^0) \times L^2(\Gamma^0)$
  is indeed the kernel of $P$.

  \begin{lemma}[parametric boundary conditions as linear constraint]
    \label{thm:projEquivalence}
    Let $i \in \{1,\dots,N\}$.
    Then $P_i$ is a well-defined linear projection operator
    such that for all $u \in H^2(\Omega(\p))$
    \begin{align}\label{eq:linearConstraint}
      P_iT_i(\p)u = P_ig_i
    \end{align}
    holds if and only if
    \begin{align}
      \label{eq:parametricConstraint}
      \exists \gamma \in \R^3 \colon
      \Biggl\{
      \begin{split}
        u|_{\Gamma_i}(x)   & = g_{i1}(\varphi(\p_i)^{-1}(x)) + \gamma_{1} x_1 + \gamma_{2} x_2 + \gamma_{3}
        \\ \partial_\nu u|_{\Gamma_i}(x) & = g_{i2}(\varphi(\p_i)^{-1}(x)) + \gamma_{1} \nu_1 + \gamma_{2} \nu_2
      \end{split}
      \qquad\forall x \in \Gamma_i(\p_i).
    \end{align}
    For $i = 0$ equation \eqref{eq:linearConstraint} is equivalent to
    $u|_{\partial\Omega} = \partial_\nu u|_{\partial\Omega} = 0$.
  \end{lemma}

  \begin{proof}
    The statement for $i=0$ is trivial, hence we only consider
    $i \in \{1,\dots,N\}$ and again drop the index $i$ for simplicity.
    Well-definedness of $P$ directly follows from the fact
    that $\eta_1,\eta_2,\eta_3 \in L^2(\Gamma^0)$ are linearly independent
    such that $\cP$ is well-defined.
    $P$ is linear, because it is the composition of linear maps.
    
    To show that $P$ is a projection, let $(v_1,v_2) \in L^2(\Gamma^0) \times L^2(\Gamma^0)$
    and set $(w_1,w_2) = P(v_1,v_2)$.
    Since $\cP$ is an $L^2(\Gamma^0)$-orthogonal projection we have
    \begin{align*}
      \cP(w_1)
      = \cP(v_1) - \cP(\cP(v_1)|_{\Gamma^0})
      = \cP(v_1) - \cP(\cP(v_1)) = 0
    \end{align*}
    and thus $P^2(v_1,v_2) = P(w_1,w_2) = (w_1, w_2) = P(v_1,v_2)$.

    Finally we show equivalence of~\eqref{eq:linearConstraint} and~\eqref{eq:parametricConstraint}.
    To this end we first note that
    $\partial_\nu \eta_1 = \nu_1$,
    $\partial_\nu \eta_2 = \nu_2$, and
    $\partial_\nu \eta_3 = 0$.
    Furthermore, since $\varphi(\p)$ is a rigid body motion in $\R^2$,
    the normal vectors $\hat{\nu}$ and $\nu$ of $\Gamma^0$ and $\Gamma(\p)$
    and corresponding normal derivatives transform according to
    \begin{align*}
      \nu(\varphi(\p)(\hat{x})) 
        &= (D \varphi(\p))(\hat{x}) \hat{\nu}(\hat{x}), &
      \partial_{\hat{\nu}}(v \circ \varphi(\p))
        &= (\partial_\nu v)\circ \varphi(\p),
    \end{align*}
    where $D\varphi(\p)$ denotes the Jacobian of the map $\varphi(\p)$.
    Using the trace operator and the formulas for $\partial_\nu \eta_k$ from above, we can write~\eqref{eq:parametricConstraint} compactly as
    \begin{align*}
      \exists \gamma \in \R^3 \colon\qquad
      ((T(\p)u) -g) \circ \varphi(\p)^{-1}
        = \sum_{k=1}^3 \gamma_k \mat{\eta_k \\ \partial_\nu \eta_k}.
    \end{align*}
    By transformation with $\varphi(\p)$ we find that this is equivalent to
    \begin{align*}
      \exists \gamma \in \R^3 \colon\qquad
        (T(\p)u) -g
        = \sum_{k=1}^3 \gamma_k \mat{\eta_k \circ \varphi(\p) \\ (\partial_\nu \eta_k) \circ \varphi(\p)}
        = \sum_{k=1}^3 \gamma_k \mat{\eta_k \circ \varphi(\p) \\ \partial_{\hat{\nu}} (\eta_k \circ \varphi(\p))}
    \end{align*}
    which can also be written as
    \begin{align*}
      (T(\p)u) -g
          \in \bigl\{ (v ,\partial_{\hat{\nu}} v) \mid v = w \circ \varphi(\p), w \in V \bigr\}
          = \bigl\{ (v ,\partial_{\hat{\nu}} v) \mid v \in V \bigr\}.
    \end{align*}
    Here we used the fact, that the transformation of the space $V$ of affine linear functions
    by $\varphi(\p)$ is $V$ itself.
    Since the right-hand-side is the kernel of $P$, the last inclusion
    is equivalent to
    \begin{align*}
      P (T(\p)u -g) = 0
    \end{align*}
    and thus~\eqref{eq:linearConstraint}.
  \end{proof}

  For notational convenience we define $Pv \colonequals (P_iv_i)_{i=0,\dots,N}$.
  The set of \emph{feasible membranes given the particle configuration $\p$} is defined by
  \begin{align*}
    U(\p) \colonequals \left\{ u \in H^2(\Omega(\p)) \mid P(T(\p)u-g) = 0 \right\}\text{.}
  \end{align*}

  We use the shorthand $J(\p,u) \colonequals J(\Omega(\p),u)$ to define the \emph{interaction energy}
  \begin{align*}
    \cJ(\p) \colonequals \min_{u \in U(\p)} J(\p,u)
  \end{align*}
  where we use the convention $\min(\emptyset) \colonequals +\infty$.
  Altogether, our model problem then reads
  \begin{align*}
    \min_{\p \in \Lambda} \cJ(\p)\text{.}
  \end{align*}

  In order to ensure well-posedness of the minimization problems
  on $U(\p)$ we will from now on assume that $g$ is smooth enough.
  Then the following lemma allows to show well-posedness using
  Lax--Milgram's theorem.
  \begin{lemma}
    \label{lemma:ellipticity}
    Let $\p \in \Lambda$.
    Then the affine subspace $U(\p)$
    can be written as $U(\p) = U_0(\p) + \hat{g}$
    for a closed subspace $U_0(\p)$ of $H^2(\Omega(\p))$
    and some $\hat{g} \in H^2(\Omega(\p))$.
    Furthermore the bilinear form
    \begin{align*}
      a(u,v) = \int_{\Omega(\p)} \kappa \Delta u \Delta v + \sigma \nabla u \cdot \nabla x \d x
    \end{align*}
    is $H^2(\Omega(\p))$-elliptic on $U_0(\p)$.
  \end{lemma}
  \begin{proof}
    First let $\hat{g} \in H^2(\Omega(\p))$ a function such that
    $T(\p)\hat{g} = g$. Then it is clear that
    \begin{align*}
      U_0(\p)
        \colonequals U(\p) - \hat{g}
        =  \left\{ u \in H^2(\Omega(\p)) \mid PT(\p)u=0 \right\}
    \end{align*}
    is a subspace of $H^2(\Omega(\p))$.
    Continuity of the trace operator $T(\p)$ and the $L^2$-projection $\cP$
    implies that $P \circ T(\p)$ is continuous and hence $U_0(\p)$ and $U(\p)$ are
    closed.

    Continuity of $a(\cdot,\cdot)$ on $U_0(\p)$ and $U(\p)$ is obvious.
    To show coercivity let $u \in U_0(\p)$. First we note that
    $u=\partial_\nu u = 0$ on $\partial \Omega$.
    Using similar arguments as in the proof of Lemma~\ref{thm:projEquivalence}
    we find that for each $i \in \{1,\dots,N\}$
    the trace of $u$ on $\Gamma_i(\p_i)$
    coincides with some affine linear function $w_i \in V$.
    Hence, in each $B_i(\p_i)$ we can extend $u$
    by the corresponding function $w_i$ to obtain a function
    $\tilde{u} \in H^2_0(\Omega)$ which coincides with $u$ in $\Omega(\p)$
    and is affine linear in each $B_i(\p_i)$.

    Using the notation $\|\cdot\|_{0,M}$ and $|\cdot|_{2,M}$
    for the $L^2$-norm and the $H^2$-half norm on the domain $M$, respectively,
    we now have
    \begin{align*}
      \|\Delta u\|_{0,\Omega(\p)}^2
      = \|\Delta \tilde{u}\|_{0,\Omega}^2
      = |\tilde{u}|_{2,\Omega}^2
      = |u|_{2,\Omega(\p)}^2.
    \end{align*}
    In the first and third step we used that all second order partial
    derivatives of $\tilde{u}|_{B_i(\p_i)} = w_i$
    vanish on $B_i(\p_i)$.
    The second step follows from $\tilde{u} \in H_0^2(\Omega)$
    (see, e.g., \cite[Lemma~5]{Graeser2015}).
    Finally we obtain for some constant $C>0$ independent of $u$
    \begin{align*}
      a(u,u)
        \geq \kappa^{-1} \|\Delta u\|_{0,\Omega(\p)}^2
        = \kappa^{-1} |u|_{2,\Omega(\p)}^2
        \geq C\kappa^{-1} \|u\|^2_{H^2(\Omega(\p))}.
    \end{align*}
    The last bound follows from the fact that $U_0(\p)$
    does not contain any nontrivial affine linear functions
    due to the boundary conditions on $\partial \Omega$
    (see, e.g., \cite[Corollary~2]{Graeser2015}).
  \end{proof}

\section{Differentiation of the reduced interaction energy}\label{sec:differentiation}
  
  In this section we investigate the differentiability of $\cJ$ on $\Lambda^\circ$.
  First we derive a technical result which shows that the admissible membrane
  sets $U(\p)$ are isomorphic and which allows us to pose our problem locally
  over a fixed reference domain $\Omega(\p)$.
  Afterwards we apply the implicit function theorem to prove differentiability
  of the reduced interaction potential and use matrix calculus to derive an
  explicit and numerically feasible expression for the first order derivatives.

  \subsection{Trace-preserving diffeomorphisms between the reference domains} 
  
    In this part we construct a local diffeomorphism between the domains
    $\Omega(\p)$ that preserves the boundary conditions.
    The basic setting is, that given a particle
    configuration $\p \in \Lambda^\circ$ we want to investigate
    the problem under changes of $\p$ along a given direction $\q \in \R^{N \times 3}$.
    To this end we will construct a diffeomorphism $\cX(\q)$
    from $\Omega(\p)$ to $\Omega(\p + \q)$ and show that is has
    the desired properties.
    The construction is based on ordinary differential equations (ODEs) and
    in particular requires the following result from ODE theory.

    For a function $\cF$ depending on multiple arguments, we denote by
    $\frac{\partial \cF}{\partial a}$ the derivative with respect to
    the argument denoted by $a$. If one of the arguments
    is a spatial coordinate in $\R^2$ we denote the $m$-th order derivative
    with respect to the spatial coordinate by $D^m$.

    \begin{lemma}
      \label{thm:propertiesODE}
      Let $\cB \subseteq \R^{N\times 3}$ be an open connected set, $m > 1$ and
      let $\cV \in C^m( [0,1] \times \cB \times \R^2, \R^2)$ be
      Lipschitz-continuous.  For $\q \in \cB$ and $x \in \R^2$ let
      $\eta(\cdot,\q,x)\colon [0,1] \rightarrow \R^2$ be the unique solution of
      the ordinary differential equation
      \begin{align*}
        \frac{\partial \eta}{\partial t}(t,\q,x) = \cV(t,\q,\eta(t,\q,x)),
        \qquad \eta(0,\q,x) = x\text{.}
      \end{align*}
      Then the map $\cX$ defined by $\cX(\q,x) \colonequals \eta(1,\q,x)$ fulfills
      $\cX \in C^m(\cB \times \R^2,\R^2)$ and is an $m$-diffeomorphism onto its
      image for all $\q \in \cB$.
      
      For all $\hat{\q} \in \R^{N\times 3}$ with
      $\eta_{\hat{\q}} \in C([0,1] \times \cB \times \R^2, \R^2)$
      as the unique solution of
      \begin{align}
        \label{eq:Xderivativeq}
        \begin{aligned}
          \frac{\partial \eta_{\hat{\q}}}{\partial t}(t,\q,x) & = \frac{\partial\cV}{\partial \q}(t,\q,\eta(t,\q,x)) \, \hat{\q} + D\cV(t,\q,\eta(t,\q,x))\eta_{\hat{\q}}(t,\q,x)
          \\ \eta_{\hat{\q}}(0,\q,x) & = 0
        \end{aligned}
      \end{align}
      holds $\partial_{\hat{\q}} \cX(\q,x) = \eta_{\hat{\q}}(1,\q,x)$.
      Also, for all $y \in \R^2$ with $\eta_y \in C([0,1] \times \cB \times \R^2, \R^2)$
      as the unique solution of
      \begin{align}
        \label{eq:Xderivativex}
        \begin{aligned}
          \frac{\partial \eta_y}{\partial t}(t,\q,x) & = D\cV(t,\q,\eta(t,\q,x)) \eta_y(t,\q,x)
          ,\qquad\eta_y(0,\q,x) = y
        \end{aligned}
      \end{align}
      holds $\partial_y \cX(\q,x) = \eta_y(1,\q,x)$.
    \end{lemma}
    
    \begin{proof}
      The global existence and uniqueness of $\eta$ is a consequence of the
      Lipschitz-continuity of $\cV$ and the well-known Picard--Lindel\"of
      theorem.  In particular, $\cX$ is well-defined.
      
      The smoothness of $\cX$ and the characterization of its derivatives is a consequence of
      \cite[Theorem 3.1, Theorem 4.1]{Hartman02}

      Concerning the inverse of $\cX(\q)$, we note that
      $\wt{\eta}(t,\q,x) \colonequals \eta(1-t,\q,x)$
      solves the equation
      \begin{align*}
        \frac{\partial\wt{\eta}}{\partial t}(t,\q,x)
        = -\cV(1-t,\q,\wt{\eta}(1,\q,x))
        ,\qquad \wt{\eta}(0,\q,x) = \eta(1,\q,x) = \cX(\q,x).
      \end{align*}
      Since $\wt{\eta}(1,\q,x) = x$, the inverse
      of $\cX(\q)$ is given by $\cX(\q)^{-1} \colonequals \wt{\eta}(1,\q,\cdot)$.
      Again, the smoothness of $\cV$ implies $m$-smoothness of $\wt{\eta}$,
      and consequently $\cX$ is an $m$-diffeomorphism.
    \end{proof}
    
    In the following we restrict ourselves to a special class of vector fields
    that is described in the result below.
    We show afterwards that the diffeomorphisms induced by such vector fields
    have a certain trace preserving property that again can be used to
    construct an isomorphism between the admissible membrane sets.
    
    \begin{lemma}
      \label{thm:existenceCV}
      Let $\p \in \Lambda^\circ$ and $m \geq 1$.
      Then there exists an open neighborhood $\cB \subseteq \R^{N\times 3}$ of
      $0 \in \R^{N\times 3}$ and a Lipschitz-continuous map
      $\cV \in C^m( [0,1] \times \cB \times \R^2, \R^2)$ such that for all
      $t \in [0,1]$, $\q \in \cB$, and $i \in \{0,\dots,N\}$ holds
      \begin{align}
        \label{eq:Vconditions}
        \begin{aligned}
          \left.\cV(t,\q,\cdot)\right|_{\Gamma_i(\p_i+t\q_i)} & = \mat{ \q_{i1} \\ \q_{i2}} + \q_{i3} \mat{ 0 & -1 \\ 1 & 0 } \left( \cdot - \mat{\p_{i1}+t\q_{i1} \\ \p_{i2}+t\q_{i2}} \right)
          \\ \left.D\cV(t,\q,\cdot)\right|_{\Gamma_i(\p_i+t\q_i)} & = \q_{i3} \mat{ 0 & -1 \\ 1 & 0 }
          \text{.}
        \end{aligned}
      \end{align}
    \end{lemma}
    
    \begin{proof}
      This is a consequence of the Whitney extension theorem, see Appendix \cref{thm:whitney}.
      It uses the fact that the $\Gamma_i$ are pairwise disjoint and that the
      right-hand-sides in \eqref{eq:Vconditions} smoothly extend to $\R^2$.
    \end{proof}

    Next we show that $\cX(\q)$ is indeed a trace preserving
    diffeomorphism from $\Omega(\p)$ to $\Omega(\p + \q)$.

    \begin{lemma}\label{thm:XExistence}
      Let $\cV$ be as in \cref{thm:existenceCV} for $m \geq 2$ and let $\cX$
      as in \cref{thm:propertiesODE} be induced by $\cV$.
      Then $\cX(0,\cdot) = \operatorname{id}_{\R^2}$, $\cX(\q,\Omega(\p)) = \Omega(\p+\q)$,
      and for all $u \in H^2(\Omega(\p))$ holds
      \begin{align}
        \label{eq:tracePreservation}
        T(\p) u = T(\p+\q)(u \circ \cX(\q)^{-1})\text{.}
      \end{align}
    \end{lemma}
    
    \begin{proof}
      From $\cV(t,0,\cdot) = 0$ follows immediately that
      $\cX(0,x) = \eta(1,0,x) = x$, \ie $\cX(0) = \operatorname{id}_{\R^2}$.
      
      We now prove \eqref{eq:tracePreservation}.
      First we show that $\cX$ preserves the boundaries.
      Let $i \in \{1,\dots,N\}$ and $x \in \Gamma_i(\p_i)$
      and define $\sigma(t) \colonequals \varphi(\p_i + t \q_i; \varphi^{-1}(\p_i;x))$.
      We will show that $\eta(t,\q,x) = \sigma(t)$ solves the ODE
      with right-hand-side $\cV$ from Lemma~\ref{thm:existenceCV}.
      To this end we will make use of the properties
      \begin{align*}
        R(\alpha+\beta) &= R(\alpha) R(\beta), &
        R(\alpha)^{-1} &= R(-\alpha),&
        R'(\alpha) &= \mat{0 & -1 \\ 1 & 0} R(\alpha)
      \end{align*}
      of the rotation matrix $R$.
      Inserting the definition of $\varphi$ and $\varphi^{-1}$ into $\sigma(t)$ we get
      \begin{align*}
        \sigma(t)
          &= R(t\q_{i3})\left(x- \mat{\p_{i1} \\ \p_{i2}} \right)
          + \mat{\p_{i1} + t\q_{i1} \\ \p_{i2} + t\q_{i2}}.
      \end{align*}
      Since we have $\sigma(t) \in \Gamma_i(\p_i + t\q_i)$
      by construction, we can use the formula~\eqref{eq:Vconditions}
      for $\cV$ on $\Gamma_i(\p_i + t\q_i)$.
      Inserting $\sigma(t)$ then gives
      \begin{align*}
        \cV(t,\q,\sigma(t))
          &= \mat{ \q_{i1} \\ \q_{i2}} + \q_{i3} \mat{ 0 & -1 \\ 1 & 0 } \left( \sigma(t) - \mat{\p_{i1}+t\q_{i1} \\ \p_{i2}+t\q_{i2}} \right)\\
          &= \mat{ \q_{i1} \\ \q_{i2}} + \q_{i3} \mat{ 0 & -1 \\ 1 & 0 } R(t\q_{i3})\left(x- \mat{\p_{i1} \\ \p_{i2}} \right)
          = \frac{\partial \sigma}{\partial t} (t).
      \end{align*}
      Hence $\eta(t,\q,x) = \sigma(t)$ solves the ODE and we have
      \begin{align}
        \label{eq:helper1709211348}
        \cX(\q,x) = \sigma(1) = \varphi\left( \p_i+\q_i; \varphi^{-1}(\p_i;x)\right)\text{.}
      \end{align}
      Recalling the convention $\q_0 \colonequals \p_0 \colonequals 0$, this is also true for $i=0$.

      In particular holds $\cX(\q,\Gamma_i(\p_i)) = \Gamma_i(\p_i+\q_i)$.
      Another immediate consequence of \eqref{eq:helper1709211348} is -- now
      recalling the definition of the rotation matrix $R$ in \eqref{eq:rotationMatrix} -- that
      \begin{align}
        \label{eq:helper1709211410}
        \nu|_{\Gamma_i(\p_i+\q_i)}(\cX(\q,x))  = R(\q_{i3})\nu|_{\Gamma_i(\p_i)}(x)
      \end{align}
      for $x \in \Gamma_i(\p_i)$.
      And, similarly, using the properties \eqref{eq:Vconditions} of $\cV$
      also allows us to compute $\frac{\partial X}{\partial x}(\q,x)$
      on $\Gamma_i(\p_i)$ from solving the ODE \eqref{eq:Xderivativex}.
      The result is
      \begin{align}
        \label{eq:helper1709211435}
        D\cX(\q,x)  = R(\q_{i3})\text{.}
      \end{align}
      Now, let $\wt{u} \colonequals u \circ \cX(\q)^{-1}$.
      From \eqref{eq:helper1709211348} we infer for $x \in \Gamma_i(\p_i+\q_i)$ that
      \begin{align}
        \label{eq:helper1709211458}
        \wt{u}(x) = u\left(\varphi(\p_i; \varphi^{-1}(\p_i+\q_i;x) \right)\text{.}
      \end{align}
      Also, from \eqref{eq:helper1709211410} and \eqref{eq:helper1709211435}
      we infer for $x \in \Gamma_i(\p_i+\q_i)$ that
      \begin{align}
        \label{eq:helper1709211459}
        \begin{aligned}
          \partial_\nu \wt{u}(x)
          & = Du( \cX^{-1}(\q;x) ) \frac{\partial \cX^{-1}(\q;x)}{\partial x} \nu|_{\Gamma_i(\p_i+\q_i)}(x)
          \\ & = Du( \cX^{-1}(\q;x) ) R(-\q_{i3}) R(\q_{i3})\nu|_{\Gamma_i(\p_i)}(\cX^{-1}(\q;x))
          \\ & = \partial_\nu u(\cX^{-1}(\q;x))
          \\ & = \partial_\nu u\left(\varphi(\p_i;\varphi^{-1}(\p_i+\q_i;x))\right)
          \text{.}
        \end{aligned}
      \end{align}
      Recalling the definition of the trace operators, \eqref{eq:traceOperator},
      we have for almost-every $x \in \Gamma_i^0$ by \eqref{eq:helper1709211458}
      \begin{align*}
        T_{i1}(\p+\q)\wt{u}(x)
        & = \wt{u}( \varphi(\p_i+\q_i;x) )
        = u\left(\varphi(\p_i;x)\right)
        = T_{i1}(\p)u(x)
      \end{align*}
      and by \eqref{eq:helper1709211459}
      \begin{align*}
        T_{i2}(\p+\q)\wt{u}(x)
        & = \partial_\nu \wt{u}( \varphi(\p_i+\q_i;x) )
        = \partial_\nu u\left( \varphi(\p_i;x) \right)
        = T_{i2}(\p)u(x)\text{.}
      \end{align*}
      Altogether this proves equation \eqref{eq:tracePreservation}.
      
      In order to show the equality $\cX(\q,\Omega(\p)) = \Omega(\p+\q)$ we define the set
      \begin{align*}
        \cZ = \left\{ (t,\eta(t,\q,x)) \mid x \in \partial\Omega(\p) \right\}\text{.}
      \end{align*}
      Now let $x_0 \in \Omega(\p)^\circ$ and assume that
      $\cX(\q,x_0) = \eta(1,\q,x_0) \notin \Omega(\p+\q)$.
      By continuity of $\cX$ this would imply that there exists a $\hat{t} \in [0,1]$
      such that $\eta(\hat{t},\q,x_0) \in \cZ$ and therefore, by definition of $\cZ$,
      there would exist a $x_1 \in \partial\Omega(\p)$ such that
      $\eta(\hat{t},\q,x_0) = \eta(\hat{t},\q,x_1)$.
      As of $x_0 \neq x_1$ this would be a contradiction to the uniqueness of $\eta$.
    \end{proof}

    \begin{lemma}\label{thm:PhiExistence}
      Let $\cX$ be as in \cref{thm:XExistence}.
      Then for all $\q \in \cB$ the map
      \begin{align*}
        \Phi(\q) \colon U(\p) & \longrightarrow U(\p+\q)
        ,\qquad u \longmapsto u \circ \cX(\q)^{-1}
      \end{align*}
      is well-defined and an isomorphism.
    \end{lemma}
    
    \begin{proof}
      From \cref{thm:XExistence} we know for every $\q \in \cB$ that the
      restriction $\cX(\q,\cdot)|_{\Omega(\p)}$ is a $2$-diffeomorphism onto
      $\Omega(\p+\q)$.  Because $\overline{\Omega(\p)}$ is compact we can
      assume without loss of generality that $\det \frac{\partial}{\partial x} \cX(\q,x)$
      is uniformly bounded, \ie that there exists a $c \in \R_{>0}$
      such that for all $\q \in \cB$ and $x \in \Omega(\p)$ holds
      $c \leq \abs{\det \frac{\partial}{\partial x} \cX(\q,x)} \leq \frac{1}{c}$.
      Elsewise we may replace $\cB$ by an appropriate sub-neighborhood.  Hence,
      \cite[Theorem 3.35]{Adams75}
      is applicable and the map
      \begin{align*}
        \wt{\Phi}(\q)\colon H^2(\Omega(\p)) & \longrightarrow H^2(\Omega(\p+\q))
        ,\qquad u \longmapsto u \circ \cX(\q)^{-1}
      \end{align*}
      is well-defined and an isomorphism, and in particular also the restriction
      $\Phi(\q) = \wt{\Phi}(\q)|_{U(\p)}$ is well-defined and an isomorphism
      onto its image.
      
      It remains to show that $\operatorname{range}(\Phi(\q)) = U(\p+\q)$.
      Suppose $u \in U(\p)$ and $\wt{u} \in U(\p+\q)$.
      Because of the trace preserving property \eqref{eq:tracePreservation}
      and by definition of $U(\p)$ and $U(\p+\q)$ it follows that
      $\Phi(\q)u = u \circ \cX(\q)\in U(\p+\q)$ and
      $\Phi(\q)^{-1} \wt{u} = \wt{u} \circ \cX(\q) \in U(\p)$, and so
      $\operatorname{range}(\Phi(\q)) = U(\p+\q)$.
    \end{proof}

  \subsection{Differentiability}
  
    In this part we use the maps $\cX$ from \cref{thm:XExistence} and $\Phi$
    from \cref{thm:PhiExistence} to transform the domain of definition for
    the functions $J(\p+\q)$ from $U(\p+\q)$ to $U(\p)$.
    Afterwards we apply the implicit function theorem to derive a differentiability result.
    
    For $\q \in \cB$ and $u \in U(\p)$ the \emph{transformed energy} is defined as
    \begin{align}
      \label{eq:transformedEnergy}
      \hat{J}(\q,u) \colonequals J(\p+\q,\Phi(\q,u))\text{,}
    \end{align}
    and the \emph{transformed reduced interaction energy} is
    \begin{align*}
      \hat{\cJ}(\q) \colonequals \min_{v \in U(\p)} \hat{J}(\q,v)
      \text{.}
    \end{align*}
    
    As in Lemma~\ref{lemma:ellipticity}
    we write the affine linear subspace $U(\p)$ as $U(\p) = U_0 + \hat{g}$
    where $U_0 \subseteq H^2(\Omega(\p))$ is a linear subspace and
    $\hat{g} \in H^2(\Omega(\p))$ is a function such that $T(\p)\hat{g} = g$.

    For notational convenience we define $D^k \colonequals \frac{\partial^k}{\partial x^k}$
    to be the differential with respect to the spatial coordinates.
    
    \begin{lemma}
      Let $\q \in \cB$ and define
      \begin{align*}
        \cA(\q,x) \colonequals \abs{\det D\cX(\q,x)} \, (D\cX(\q,x))^{-1}  (D\cX(\q,x))^{-T}
        \text{.}
      \end{align*}
      It holds
      \begin{align}
        \label{eq:transformedIntegrand}
        \hat{J}(\q,u) = \frac{1}{2}\int_{\Omega(\p)} \kappa \frac{\div\left( \cA(\q) \nabla u\right)^2}{\abs{\det D\cX(\q)}} + \sigma \norm{\nabla u}_{\cA(\q)}^2 \d{x}
      \end{align}
      and $\hat{J}_u \in C^{m-2}(\cB \times U(\p), U_0')$.
    \end{lemma}
    
    \begin{proof}
      Equation \eqref{eq:transformedIntegrand} is a direct application of
      \cref{thm:transformationDerivatives} applied to $X = \cX(\q)$.
      
      Furthermore, for all $v \in H^2(\Omega)$ we have
      \begin{align*}
        \hat{J}_u(\q,u;v)
        & = \int_{\Omega(\p)} \kappa \frac{\div(\cA(\q)\nabla u) \, \div(\cA(\q) \nabla v)}{\abs{\det D\cX(\q)}} + \sigma \cA(\q)\nabla u \cdot \nabla v \d{x}
        \text{.}
      \end{align*}
      As of $\cX \in C^m$ we know that $D\cX(\q)$ and $D^2\cX(\q)$ are both
      $(m-2)$ times continuously differentiable.  Because the $\cX(\q)$ are
      diffeomorphisms with $\cX(0)$ and because $\cX$ is continuous we have
      $\det(D\cX(\q)) > 0$.  Because $\overline{\Omega(\p)}$ is compact, we can
      assume without loss of generality that there exists a $c \in \R_{>0}$
      such that $\det(D\cX(\q)) > c$, else we replace $\cB$ by a suitable
      sub-neighborhood.  Consequently, the integrand of $\hat{J}_u$ is $(m-2)$
      times differentiable with respect to $\q$.  Moreover, the integrand is
      even smooth with respect to $u$ and hence application of the dominated
      convergence theorem yields $\hat{J}_u \in C^{m-2}(\cB \times U(\p), U_0')$.
    \end{proof}

    \begin{lemma}
      There exists a neighborhood $\hat{\cB}$ of $0 \in \R^{N\times 3}$
      such that $\cJ \in C^{m-2}(\hat{\cB}+\p)$
      and for all $\q \in \hat{\cB}$ and multi-indices $\alpha$ with
      $\abs{\alpha} \leq m-2$ holds
      \begin{align*}
        \frac{\partial^\alpha}{\partial \p^\alpha} \cJ(\p+\q)
        = \frac{\partial^\alpha}{\partial \q^\alpha} \hat{\cJ}(\q)\text{.}
      \end{align*}
      In particular, if $m \geq 3$ and $u = \argmin_{v \in U(\p)} J(\p,v)$ then
      \begin{align}
        \label{eq:firstDerivative}
        \dd{\p} \cJ(\p) = \dd{\q} \hat{J}(0,u)
        \text{.}
      \end{align}
    \end{lemma}
    
    \begin{proof}
      Let $\hat{J}_u \colonequals \dd{u} \hat{J}$ and
      $\hat{J}_{uu} \colonequals \frac{\partial^2}{\partial u^2} \hat{J}$.
      Define
      \begin{align*}
        F\colon \cB \times U_0 & \longrightarrow U_0'
        ,\qquad (\q,v) \longmapsto \hat{J}_u(\q,v + \hat{g})
        \text{.}
      \end{align*}
      Suppose that $u \in U(\p)$ is the unique solution of $\min_{v \in U(\p)} J(\p,v)$,
      and define $\hat{u} \colonequals u - \hat{g}$.
      Then by \eqref{eq:transformedEnergy} and because of $\Phi(0)u = u$ it also follows
      that $u$ is the unique minimizer of $\hat{J}(0,\cdot)$ over $U(\p)$ and therefore
      and therefore
      \begin{align*}
        F(0,\hat{u}) = \hat{J}_u(0, \hat{u} + \hat{g}) = \hat{J}_u(0,u) = 0 \in U_0'
        \text{.}
      \end{align*}
      Moreover, for all $v,w \in U_0$ holds
      \begin{align*}
        F_u(0,\hat{u};v,w)
        & = \hat{J}_{uu}(0,u;v,w)
        = \int_{\Omega(\p)} \kappa \Delta v \, \Delta w + \sigma \nabla v \cdot \nabla w \d{x}\text{.}
      \end{align*}
      This defines an elliptic bilinear form over $U_0$ (cf. Lemma~\ref{lemma:ellipticity})
      and hence $F_u(0,\hat{u})$
      is invertible in $U_0$ by virtue of Lax--Milgram's theorem.
      Application of the implicit function theorem, \cref{thm:IFT}, yields a neighborhood
      $\hat{\cB} \subseteq \cB$ of $0$ and a function $\hat{\u} \in C^{m-2}(\hat{\cB}, U_0)$
      such that $\hat{\u}(0) = \hat{u}$ and $F(0,\hat{\u}(\q)) = 0$ for all $\q \in \hat{\cB}$.
      In particular, $\hat{\cJ}(\q) = \hat{J}(\q,\hat{\u}(\q)+\hat{g})$ for all $\q \in \cB$.
      From $\cJ(\p+\q) = \hat{\cJ}(\q)$ we infer
      \begin{align*}
        \frac{\partial^\alpha}{\partial \p^\alpha} \cJ(\p)
        = \frac{\partial^\alpha}{\partial \q^\alpha} \hat{\cJ}(\q)
      \end{align*}
      for all multi-indices $\alpha$ with $\abs{\alpha} \leq m-2$.
      For $m\geq 3$ this in particularly implies
      \begin{align*}
        \dd{\p} \cJ(\p)
        & = \dd{\q} \hat{J}(0,\hat{\u}(0)+\hat{g}) + \hat{J}_{u}(0,\hat{\u}(0)+\hat{g}) \, \dd{\q} \hat{\u}(0)
        = \dd{\q} \hat{J}(0,u)
        \text{.}
      \end{align*}
      
    \end{proof}

  \subsection{A numerically feasible representation of the first derivative}

    In the following paragraphs we discuss a way to derive a numerically feasible
    expression for the first order derivative $\partial_{\e} \cJ(\p)$ of the
    reduced interaction energy $\cJ$ in $\p \in \Lambda^\circ$ in direction of
    an $\e \in \R^{N\times 3}$.

    An important component of the integrand's derivative is $\partial_{\e} \cX(0)$
    and its spatial derivatives.  From \eqref{eq:Xderivativeq} we know
    that this derivative can be evaluated by solving an ODE.  This is not
    practically feasible, however, because those computations would be too
    expensive.  Besides, it also requires knowledge of the vector field
    $\cV$, which may be hard to construct explicitly.  Instead we restrict ourselves
    to a subclass of vector fields in the sense of \cref{thm:existenceCV} for
    which $\partial_{\e} \cX(0)$ is can be computed easily from information
    that is available a-priori.
    
    To this end, suppose a vector field $V \colon \Omega(\p) \rightarrow \R^2$
    such that for $i \in \{0,\dots,N\}$
    \begin{align}
      \label{eq:V0conditions}
      \begin{aligned}
        V|_{\Gamma_i(\p_i)} & = \mat{ \e_{i1} \\ \e_{i2}} + \e_{i3} \mat{ 0 & -1 \\ 1 & 0 } \left( \cdot - \mat{\p_{i1}\\ \p_{i2}} \right)
        \\ DV|_{\Gamma_i(\p_i)} & = \e_{i3} \mat{ 0 & -1 \\ 1 & 0 }
      \end{aligned}
    \end{align}
    where we again use the convention $\e_0 \colonequals 0$.
    Usually it is easy to construct such a $V$ in a way that it is also
    numerically accessible.
    Next we extend this to a vector field $\cV$ such that \eqref{eq:Vconditions}
    is fulfilled, where we can make the simplifying assumption that $\cB$
    is a ball of radius $r \in \R_{>0}$, and such that for all $t \in [0,1]$
    and $\lambda \in (0,r)$ the scaling properties
    \begin{align}
      \label{eq:scalingCondition}
      \begin{aligned}
        \cV(t,\lambda\e,x) & = \frac{\lambda}{r} \cV\left(\frac{\lambda}{r} t,r\e,x\right)
        \\ \cV(0,r\e,x) & = rV(x)
      \end{aligned}
    \end{align}
    hold.

    In view of \eqref{eq:Xderivativeq} and given $x \in \Omega(\p)$,
    we have $\partial_{\e} \cX(0,x) = \eta_{\e}(1,x)$ where $\eta_{\e}(\cdot,x)$ solves the ODE
    \begin{align*}
      \frac{\partial \eta_{\e}}{\partial t}(t,x) & = \partial_{\e} \cV(t,0,\eta(t,0,x)) + D\cV(t,0,\eta(t,0,x)) \eta_{\e}(t,x)
      ,\qquad \eta_{\e}(0,x) = 0\text{.}
    \end{align*}
    As of $\cV(t,0,\cdot) \equiv 0$ we have $\eta(t,0,x) = x$ and $D\cV(t,0,\eta(t,0,x)) = 0$.
    Furthermore, from \eqref{eq:scalingCondition} we are able to conclude
    \begin{align*}
      \partial_{\e} \cV(t,0,x)
      & = \lim_{\lambda\searrow 0} \frac{\cV(t,\lambda\,\e,x) - \cV(t,0,x)}{\lambda}
      = \lim_{\lambda\searrow 0} \frac{\cV\left(\frac{\lambda}{r}\,t,r\e,x\right)}{r}
      = V(x)\text{.}
    \end{align*}
    Therefore, $\eta_{\e}$ is the solution of the ODE
    \begin{align*}
      \frac{\partial \eta_{\e}}{\partial t}(t,x) = V(x), \qquad \eta_{\e}(0,x) = 0\text{,}
    \end{align*}
    which implies $\eta_{\e}(t,x) = t\,V(x)$ and hence also $\partial_{\e} \cX(0) = V$.

    When computing the derivative, we will make use of the following identities
    from matrix calculus.
    \begin{lemma}
      Suppose $M \in C^1( \R^d, \R^{n\times n} )$ and that $M(x)$ is invertible
      for all $x \in \R^d$.
      Then
      \begin{align}
        \dd{x_i} \det(M(x)) & = \det(M) \Tr\left( M(x)^{-1} \dd{x_i} M(x)\right) \label{eq:matCalc1}
        \\ \dd{x_i} M(x)^{-1} & = -M(x)^{-1} \frac{\partial M(x)}{\partial x_i} M(x)^{-1} \label{eq:matCalc2}
        \\ \dd{x_i} \Tr(M(x)) & = \Tr\left(\dd{x_i} M(x) \right) \label{eq:matCalc3}
        \text{.}
      \end{align}
    \end{lemma}
    
    \begin{proof}
      See literature on matrix calculus, \eg \cite[Chapter 9]{Lax97}.
    \end{proof}

    \begin{lemma}
      Let $V \colonequals \partial_{\e} \cX(0)$, $u \colonequals \argmin_{v \in U(\p)} J(\p,u)$,
      and
      \begin{align*}
        \cA'(0) \colonequals \div(V)I - DV - DV^T
        \text{.}
      \end{align*}
      Then
      \begin{align}
        \label{eq:volumeDerivative}
        \begin{aligned}
          \partial_{\e} \cJ(\p)
          & = \int_{\Omega(\p)} \kappa \Delta u \left( \cA'(0) : D^2u - \Delta V \cdot \nabla u - \frac{1}{2} \div(V)\Delta u \right) \d{x}
          \\ & \qquad + \int_{\Omega(\p)} \frac{\sigma}{2} \norm{\nabla u}_{\cA'(0)}^2 \d{x}
          \text{.}
        \end{aligned}
      \end{align}
    \end{lemma}
    
    \begin{proof}
      From \eqref{eq:firstDerivative} we know that
      $\partial_{\e} \cJ(\p) = \partial_\e \hat{J}(0,u)$ and hence it
      suffices to compute the latter.
      In the following we use without further emphasis the identities
      $D\cX(0,\cdot) \equiv \operatorname{id}_{\R^2}$ and $\det D\cX(0) \equiv 1$.
      Based on this and on the identities \eqref{eq:matCalc1} and \eqref{eq:matCalc2}
      we have
      \begin{align}
        \label{eq:helper1707071339}
        \begin{split}
          \left. \dd{\e} \det(D\cX(\q)) \right|_{\q = 0}
          & = \left. \det(D\cX(\q)) \Tr\left( D\cX(\q)^{-1} \dd{\e} D\cX(\q) \right) \right|_{\q = 0}
          \\ & = \Tr\left(\dd{\e} D\cX(\q) \right) = \div(V)
        \end{split}
        \\
        \label{eq:helper1707071345}
        \begin{split}
          \left. \dd{\e} D\cX^{-1}(\q) \right|_{\q = 0}
          & = \left. -D\cX(\q)^{-1} \frac{\partial D\cX(\q)}{\partial \e} D\cX(\q)^{-1}\right|_{\q=0}
          \\ & = - \frac{\partial D\cX(0)}{\partial \e} = -DV
        \text{.}
        \end{split}
      \end{align}
    
      By definition of $\cA$ we have
      \begin{align*}
        \cA(\q) = \det(D\cX(\q)) \, D\cX(\q)^{-1} D\cX(\q)^{-T}\text{,}
      \end{align*}
      where we again used $\det(\cX(\q)) > 0$.
      The product rule together with the identities \eqref{eq:helper1707071339}
      and \eqref{eq:helper1707071345} then leads us to
      \begin{align}
        \label{eq:helper1707071323}
        \begin{aligned}
          \left.\dd{\e} \cA(\p) \right|_{\q=0}
          & = \left.\frac{\partial \det(D\cX(\q))}{\partial \e}\right|_{\q=0} D\cX(\q)^{-1} D\cX(\q)^{-T}
          \\ & \qquad + \det(D\cX(\q)) \left.\frac{\partial D\cX(\q)^{-1}}{\partial \e}\right|_{\q=0} D\cX(\q)^{-T}
          \\ & \qquad + \det(D\cX(\q)) D\cX(\q)^{-1} \left.\frac{\partial D\cX(\q)^{-T}}{\partial \e}\right|_{\q=0}
          \\ & = \div(V) I - DV - DV^T
          \\ & = \cA'(0)
          \text{.}
        \end{aligned}
      \end{align}
      Equation \eqref{eq:helper1707071323} gives us, using $\partial_i \colonequals \dd{x_i}$,
      \begin{align*}
        \partial_i \cA'(0)
        & = \left(\sum_{k=1}^{2} \partial_{ik} V_k \right) I - \partial_i DV - \partial_i DV^T
      \end{align*}
      and so
      \begin{align}
        \label{eq:helper1707071431}
        \begin{aligned}
          \sum_{i=1}^{2} \partial_i \cA'_{ij}(0)
          & = \sum_{i,k=1}^{2} \partial_{ik} V_k \delta_{ij} - \sum_{i = 1}^{2} \left( \partial_{ij} V_i + \partial_{ii} V_j \right)
          \\ & = \sum_{k=1}^{2} \partial_{jk} V_k - \sum_{i = 1}^{2} \left( \partial_{ij} V_i + \partial_{ii} V_j \right)
          \\ & = \sum_{i=1}^{2} \left( \partial_{ji} V_i - \partial_{ij} V_i - \sum_{i=1}^{2} \partial_{ii} V_j \right)
          \\ & = -\Delta V_j
        \end{aligned}
      \end{align}
      for $j\in\{1,2\}$.
      Noting
      \begin{align*}
        \div(\cA(\q) \nabla u)
        & = \sum_{i=1}^{2} \partial_i \left(\sum_{j=1}^{2} \cA(\q)_{ij} \partial_j u \right)
        \\ & = \sum_{i,j=1}^{2} \left( \cA(\q)_{ij} \partial_{ij} u + \partial_i\cA(\q)_{ij} \partial_j u \right)
        \\ & = \cA(\q) : D^2u + \sum_{j=1}^{2} \left(\sum_{i=1}^{2} \partial_i \cA(\q)_{ij} \right) \partial_j u
      \end{align*}
      we therefore conclude
      \begin{align}
        \label{eq:helper1707071438}
        \begin{aligned}
          \left.\dd{\e} \div(\cA(\q)\nabla u)\right|_{\q=0}
          & = \cA'(0) : D^2 u - \Delta V \cdot \nabla u
        \end{aligned}
      \end{align}
      where $\Delta V \colonequals (\Delta V_i)_{i=1,\dots,n}\in \R^{2}$.
      
      We recall
      \begin{align*}
        \hat{J}(\q,u)
        & = \frac{1}{2}\int_{\Omega(\p)} \kappa \frac{\div(\cA(\q)\nabla u)^2}{\det D\cX(\q)} + \sigma \cA(\q)\nabla u \cdot \nabla u \d{x}
        \text{.}
      \end{align*}
      Combining \eqref{eq:helper1707071339}, \eqref{eq:helper1707071323}
      and \eqref{eq:helper1707071438} yields
      \begin{align*}
        \left.\dd{\e} \hat{J}(\q,u)\right|_{\q=0}
        & = \frac{1}{2}\int_{\Omega(\p)} 2\kappa \div(\cA(\q)\nabla u) \left.\dd{\e} \div(\cA(\q) \nabla u)\right|_{\q=0} \d{x}
        \\ & \quad - \frac{1}{2} \int_{\Omega(\p)} \kappa \frac{\div(\cA(\q)\nabla u)^2}{\left(\det D\cX(\q)\right)^2} \left.\dd{\e} \det( D\cX(\q) )\right|_{\q=0} \d{x}
        \\ & \qquad + \frac{1}{2} \int_{\Omega(\p)} \sigma \left.\dd{\e} \cA(\q)\right|_{\q = 0} \nabla u \cdot \nabla u \d{x}
        \\ & = \int_{\Omega(\p)} \kappa \Delta u \left( \cA'(0) : D^2u - \Delta V \cdot \nabla u - \frac{1}{2} \div(V)\Delta u \right)
        \\ & \qquad + \frac{\sigma}{2} \norm{\nabla u}_{\cA'(0)}^2 \d{x}
        \text{,}
      \end{align*}
      which proves the claim as stated.
    \end{proof}

    In general, the exact minimizer for $\cJ(\p)$ is unknown and can only be approximated.
    The following result gives an upper bound on the approximation error.
    \begin{lemma}
      Let $u = \argmin_{v \in U(\p)} J(\p,v)$ and $\wt{u} \in H^2(\Omega(\p))$.
      Then there exists a constant $C>0$ such that
      \begin{align*}
        \abs{ \partial_\e \hat{J}(0,u) - \partial_\e \hat{J}(0,\wt{u})} \leq C \norm{V}_{C^2(\Omega(\p))} \, \norm{u + \wt{u}}_{H^2(\Omega(\p))} \, \norm{u - \wt{u}}_{H^2(\Omega(\p))}
      \end{align*}
    \end{lemma}
    
    \begin{proof}
      Note from \eqref{eq:volumeDerivative} that $\hat{J}_u$ is induced by
      a non-symmetric bilinear form, \ie there exists a bilinear form
      $a\colon H^2(\Omega(\p)) \times H^2(\Omega(\p)) \rightarrow \R$
      such that $\hat{J}_u(v) = a(v,v)$.
      Upon investigation of the coefficients of $a$ it is readily seen
      that there exists a $C \in \R_{>0}$ such that for all $v,w \in H^2(\Omega(\p))$
      \begin{align*}
        \abs{a(v,w)} \leq C \norm{V}_{C^2(\Omega(\p))} \, \norm{v}_{H^2(\Omega(\p))} \norm{w}_{H^2(\Omega(\p))}\text{.}
      \end{align*}
      Now, consider
      \begin{align*}
        \abs{ \partial_{\e}\hat{J}(0,u) - \partial_{\e}\hat{J}(0,\wt{u})}
        & = \abs{ a(u,u) - a(\wt{u},\wt{u}) }
        \\ & = \abs{ \frac{1}{2} a(u+\wt{u},u-\wt{u}) + \frac{1}{2} a(u-\wt{u},u+\wt{u}) }
        \\ & \leq C \norm{V}_{C^2(\Omega(\p))} \, \norm{u + \wt{u}}_{H^2(\Omega(\p))} \, \norm{u - \wt{u}}_{H^2(\Omega(\p))}
        \text{.}
      \end{align*}
    \end{proof}
    It is important to note that while the derivative $\dd{\e} \cJ(\p)$
    itself is independent of $V$, the actual choice of $V$ very well
    enters the approximation error.
    Therefore it is desirable to construct a $V$ with a bounded $C^2$-norm.
  
\section{Numerical Examples}\label{sec:examples}

  In this section we illustrate our formula for the derivative by numerical computations for various particle configurations.
  Here we always defined the bending rigidity $\kappa = 1$ and specified no surface tension $\sigma = 0$.
  The optimal membrane shapes $u(\p)$ for fixed particle configurations $\p$ were approximated by finite element discretizations $u_h(\p) \approx u(\p)$.
  A possible discretization using a penalty approach is discussed in \cite{GrKi17} where we also gave a proof of convergence.
  The vector fields $V$ that occur in the derivative were explicitly constructed in such a way that they both fulfill \eqref{eq:V0conditions} and can be represented as finite element functions within the used discretization.
  The expressions for the discretized derivatives were evaluated exactly by using standard quadrature methods.

  \subsection{Two circular particles}
  Let $\Omega = [-10,10]^2$ and consider two circular particles of radius one $B_1$, $B_2$ that each induce on $\Gamma_i \colonequals \partial B_i$ the boundary conditions
  \begin{align*}
    u|_{\Gamma_i}(y) & = 0 + \gamma_{i1} y_1 + \gamma_{i2} y_2 + \gamma_{i3}
    , \qquad \partial_\nu u|_{\Gamma_i}(y) = 1 + \gamma_{i1}\nu_1(y) + \gamma_{i2}\nu_2(y)
    \text{.}
  \end{align*}
  We are interested in the interaction energy as a function of the distance of these two particles and therefore we define
  \begin{align*}
    \hat{\q} & \colonequals \mat{ -1 & 0 & 0 \\ 1 & 0 & 0 }
    ,\qquad
    \q \colonequals \frac{\hat{\q}}{\norm{\hat{\q}}}
    ,\qquad
    f(r)  \colonequals \cJ(r\q)
    \text{.}
  \end{align*}
  In this formulation $r$ is the distance between the particle centers and the particles touch for $r=2$.
  
  \begin{figure}
    \begin{center}
      \includegraphics[width=0.45\textwidth]{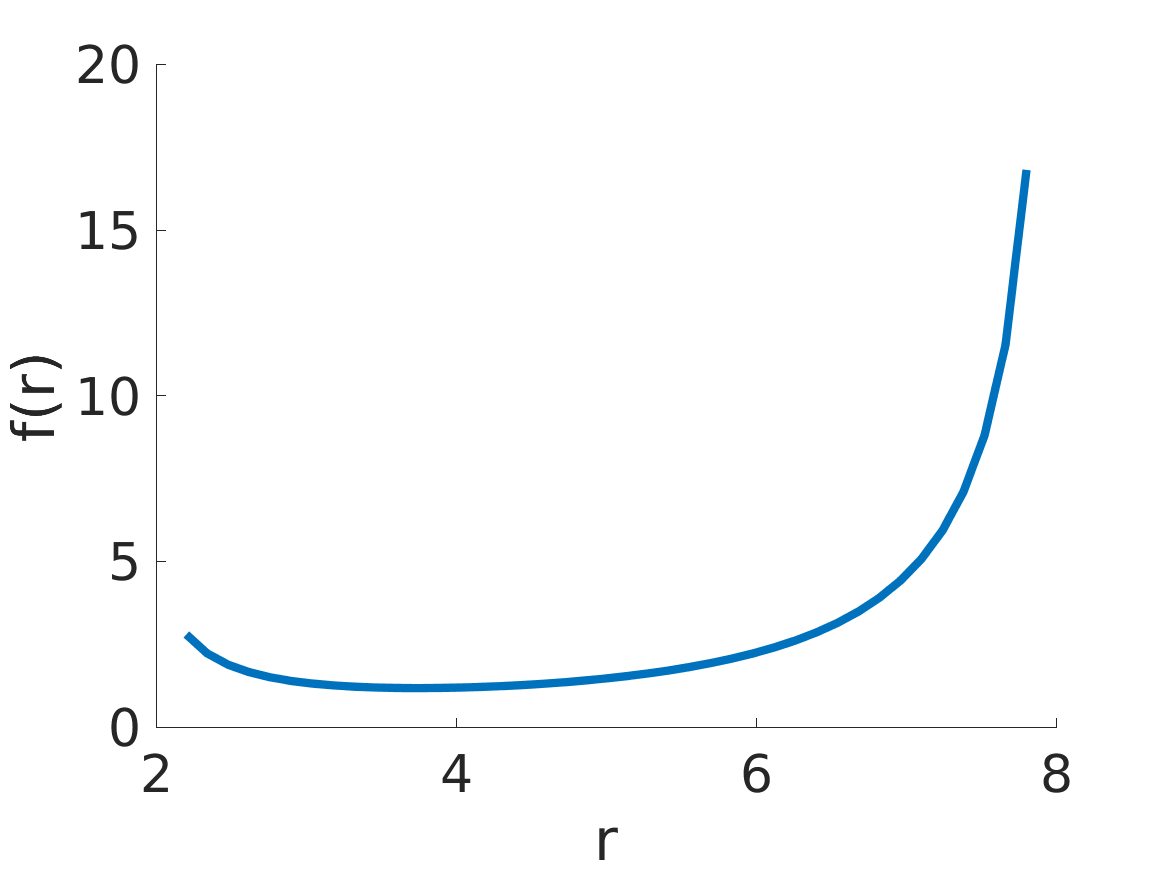}
      \ 
      \includegraphics[width=0.45\textwidth]{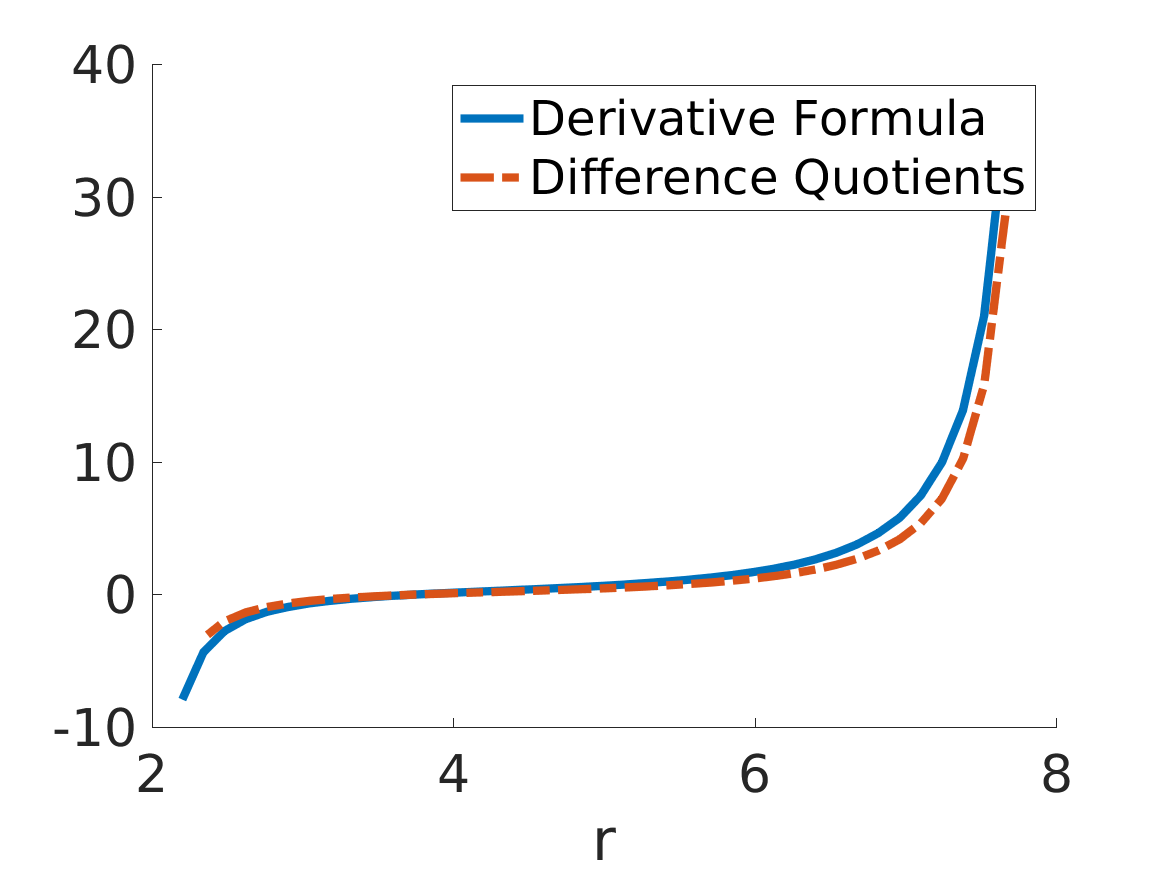}
    \end{center}
    \caption{Left: Elastic energy for two circular inclusions as a function of their distance. Right: Comparison of difference quotients with the derivative formula \eqref{eq:volumeDerivative} for this function.}
    \label{fig:twoCircularParticles}
  \end{figure}
  
  On the left picture of \cref{fig:twoCircularParticles} we depict the approximate values of $f(r)$ for $2.06 \leq r \leq 7.94$ that we obtained from our discretization, and
  on right we show two approximations of $f'(r)$.
  One approximation was obtained by computing the difference quotients from the function values and the other one was obtained from the derivative formula \eqref{eq:volumeDerivative}.

  \subsection{Two peanut shaped particles}
  Let $\Omega = [-5,5]^2$ and consider particles whose shape is defined by the zero level set of
  \begin{align*}
    \frac{1}{20} - x^4 + \frac{19}{20} x^2 - 2x^2 y^2- \frac{19}{20} y^2 - y^4\text{.}
  \end{align*}
  We assume that each particle induces the following boundary conditions:
  \begin{align*}
    u|_{\Gamma_i}(y) & = 0 + \gamma_{i1} y_1 + \gamma_{i2} y_2 + \gamma_{i3}
    , \qquad \partial_\nu u|_{\Gamma_i}(y) = \partial_\nu g(y) + \gamma_{i1}\nu_1(y) + \gamma_{i2}\nu_2(y)
  \end{align*}
  where $g(y) \colonequals \frac{1}{2} (y_1^2 + y_2^2)$.
  Since the particles are not rotationally
  symmetric, the interaction energy might not just
  depend on the particle distance,
  but on the distances in $x$- and $y$-direction
  and the relative rotation angle.
  To investigate all three directions
  we define
  \begin{align*}
    \p      & = \mat{-2.5 & 0 & 0 \\ 2.5 & 0 & 0}
    ,\quad\q^1 = \mat{1 & 0 & 0 \\ 0 & 0 & 0}
    ,\quad\q^2 = \mat{0 & 1 & 0 \\ 0 & 0 & 0}
    ,\quad\q^3 = \mat{0 & 0 & 1 \\ 0 & 0 & 0}
  \end{align*}
  and $f_i(t) \colonequals \cJ(\p+t\q^i)$.
  Then $f_1,f_2$, and $f_3$ represent the interaction energy
  along changes of the relative distance in $x$- and $y$-direction
  and of the relative rotation, respectively.
  
  \begin{figure}
    \begin{center}
      \includegraphics[width=0.45\textwidth]{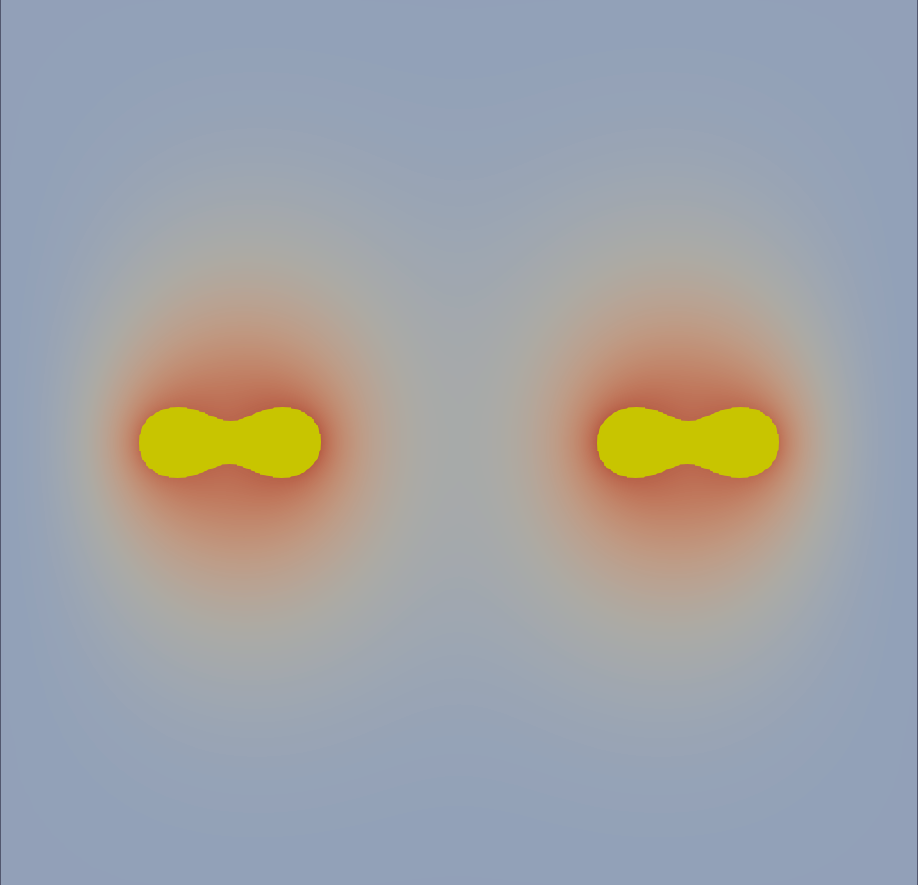}
      \hspace{2em}
      \includegraphics[width=0.45\textwidth]{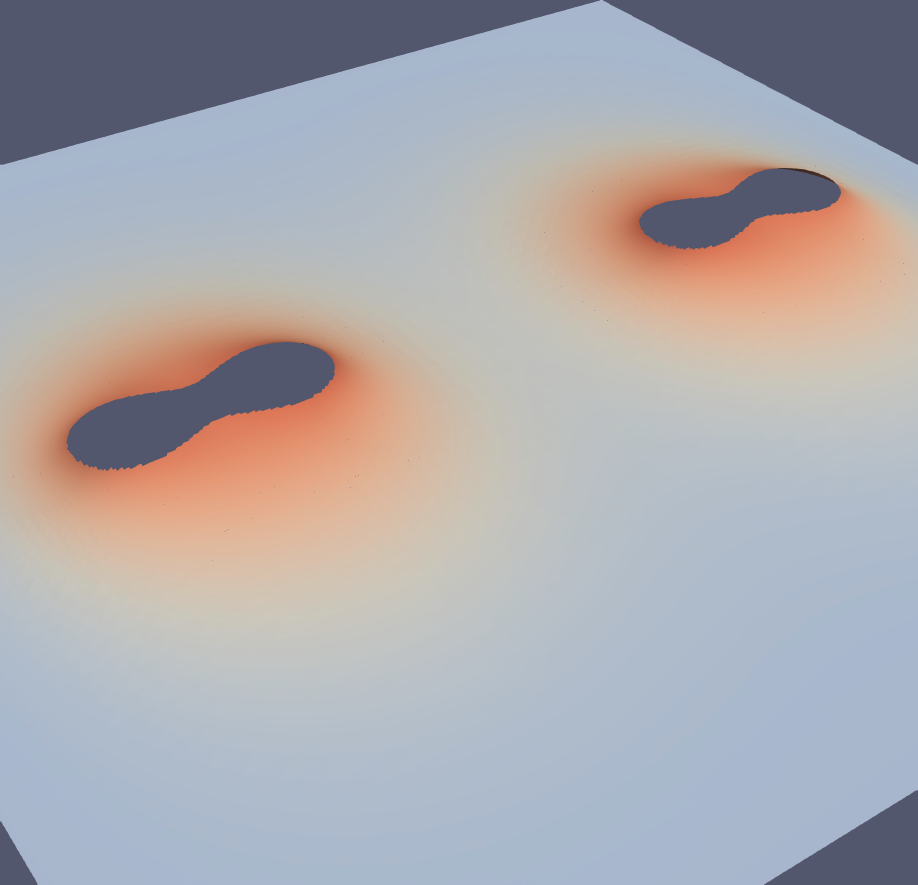}
    \end{center}
    \caption{Left: Level set view of the optimal membrane shape for the particle configuration $\p$. Right: Rendered 3D view.}
    \label{fig:peanuts0}
  \end{figure}
  
  \begin{figure}
    \begin{center}
      \includegraphics[width=0.45\textwidth]{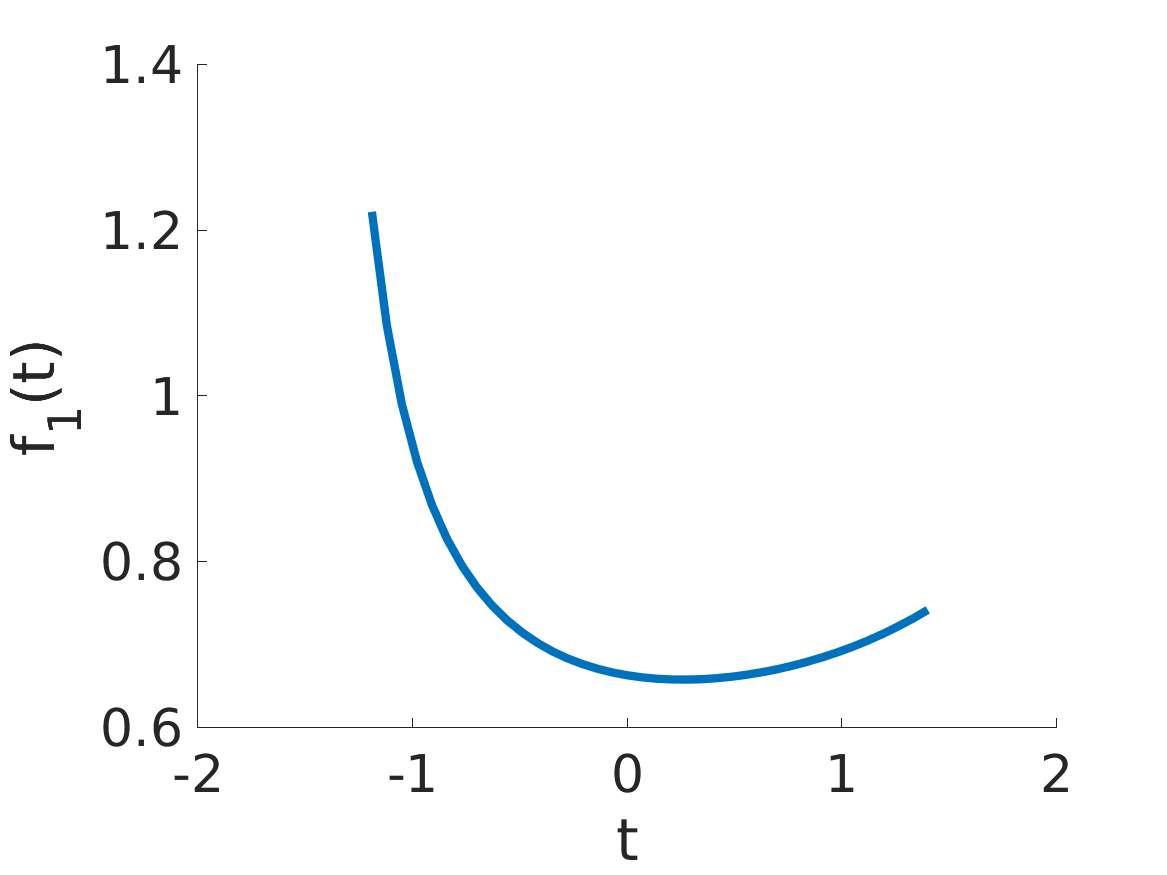}
      \ 
      \includegraphics[width=0.45\textwidth]{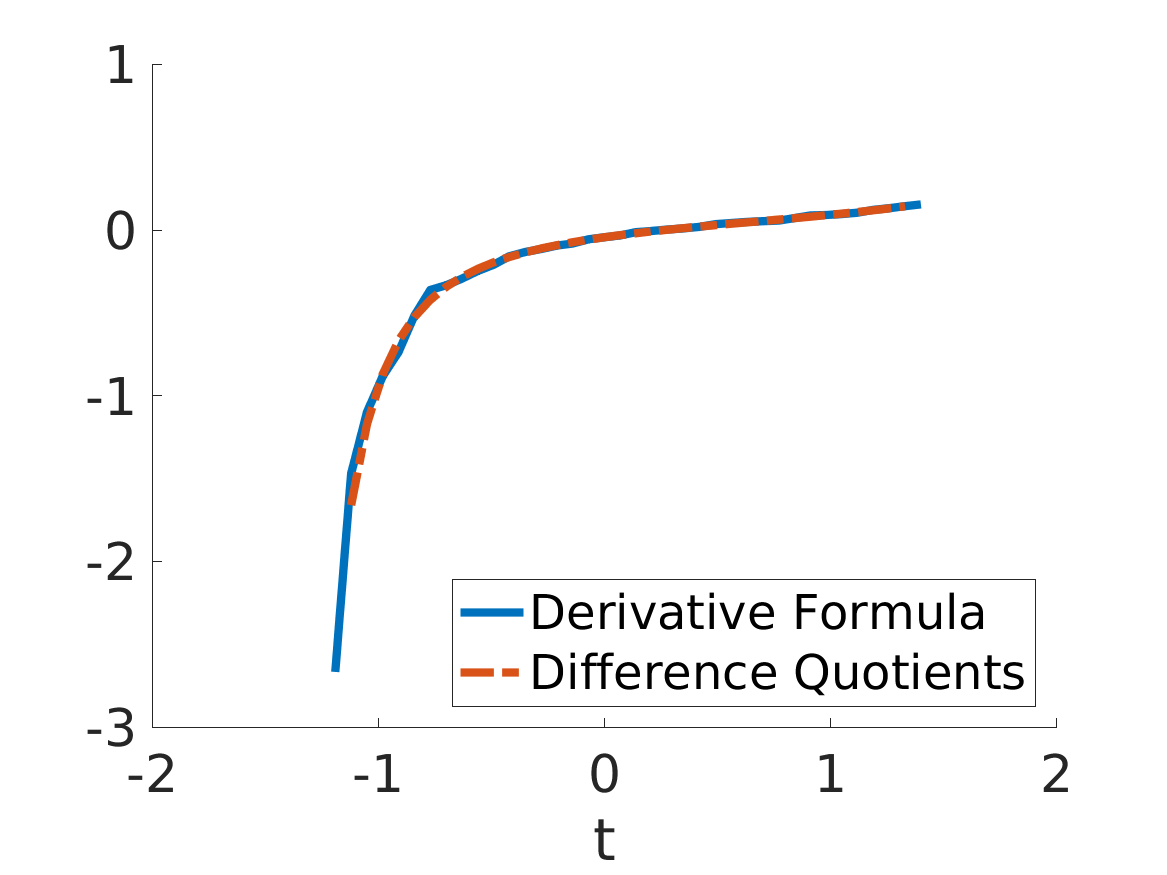}
    \end{center}
    \caption{Left: Plot of $f_1(t)$. Right: Comparison of difference quotients with the derivative formula \eqref{eq:volumeDerivative}.}
    \label{fig:peanuts1}
  \end{figure}
  
  \begin{figure}
    \begin{center}
      \includegraphics[width=0.45\textwidth]{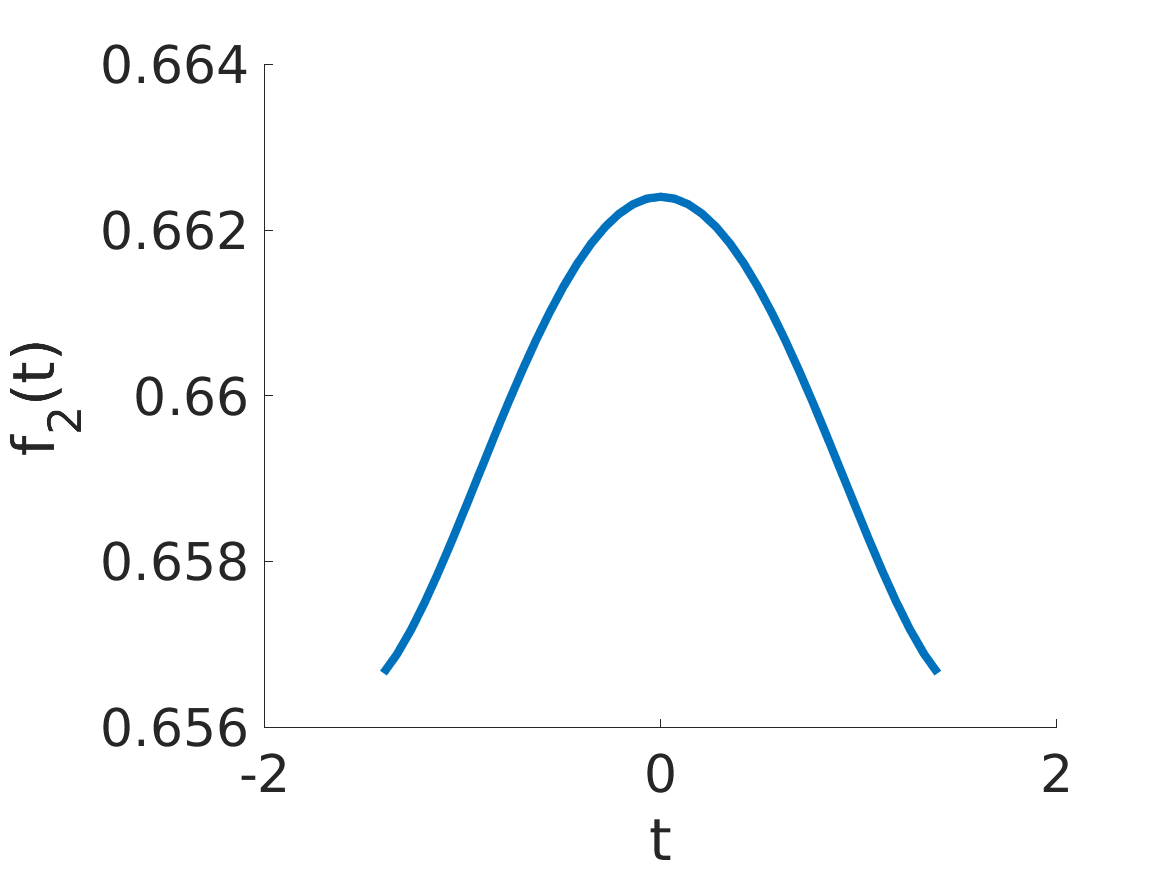}
      \ 
      \includegraphics[width=0.45\textwidth]{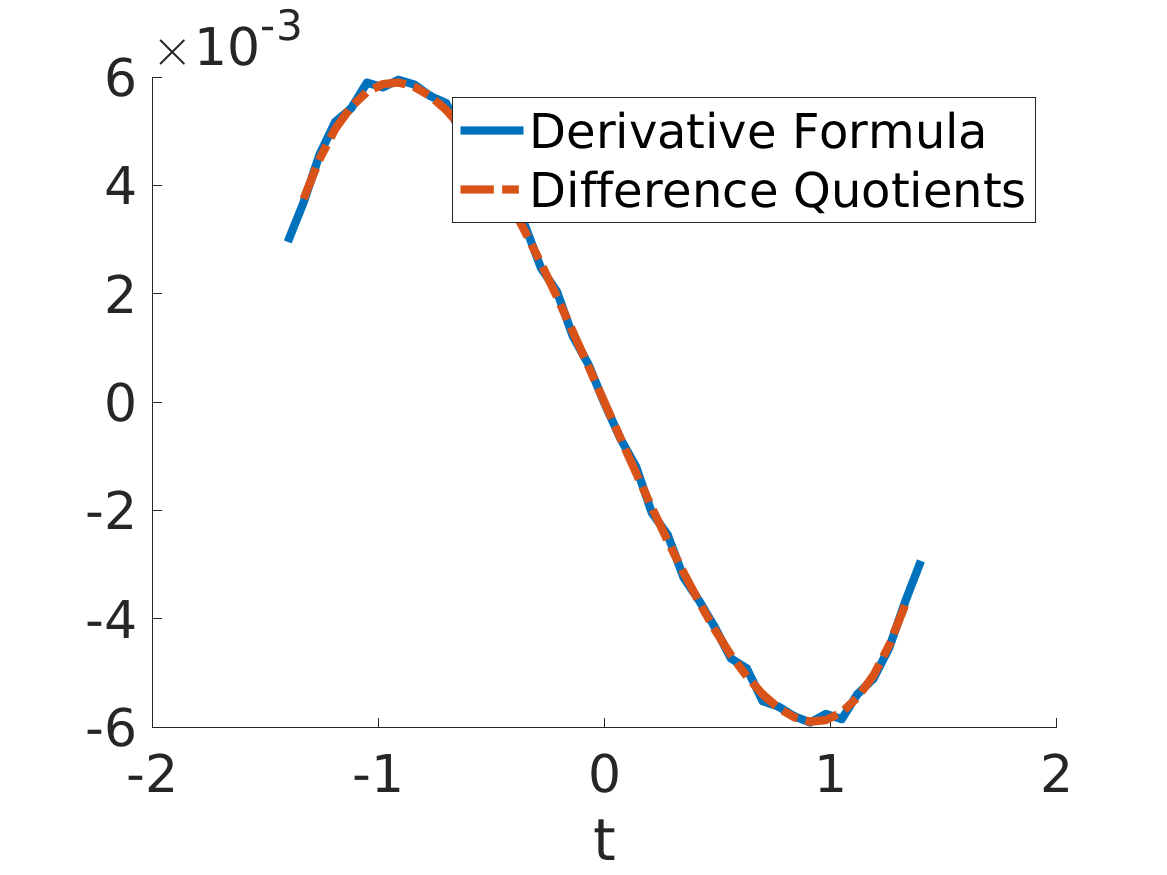}
    \end{center}
    \caption{Left: Plot of $f_2(t)$. Right: Comparison of difference quotients with the derivative formula \eqref{eq:volumeDerivative}.}
    \label{fig:peanuts2}
  \end{figure}
  
  \begin{figure}
    \begin{center}
      \includegraphics[width=0.45\textwidth]{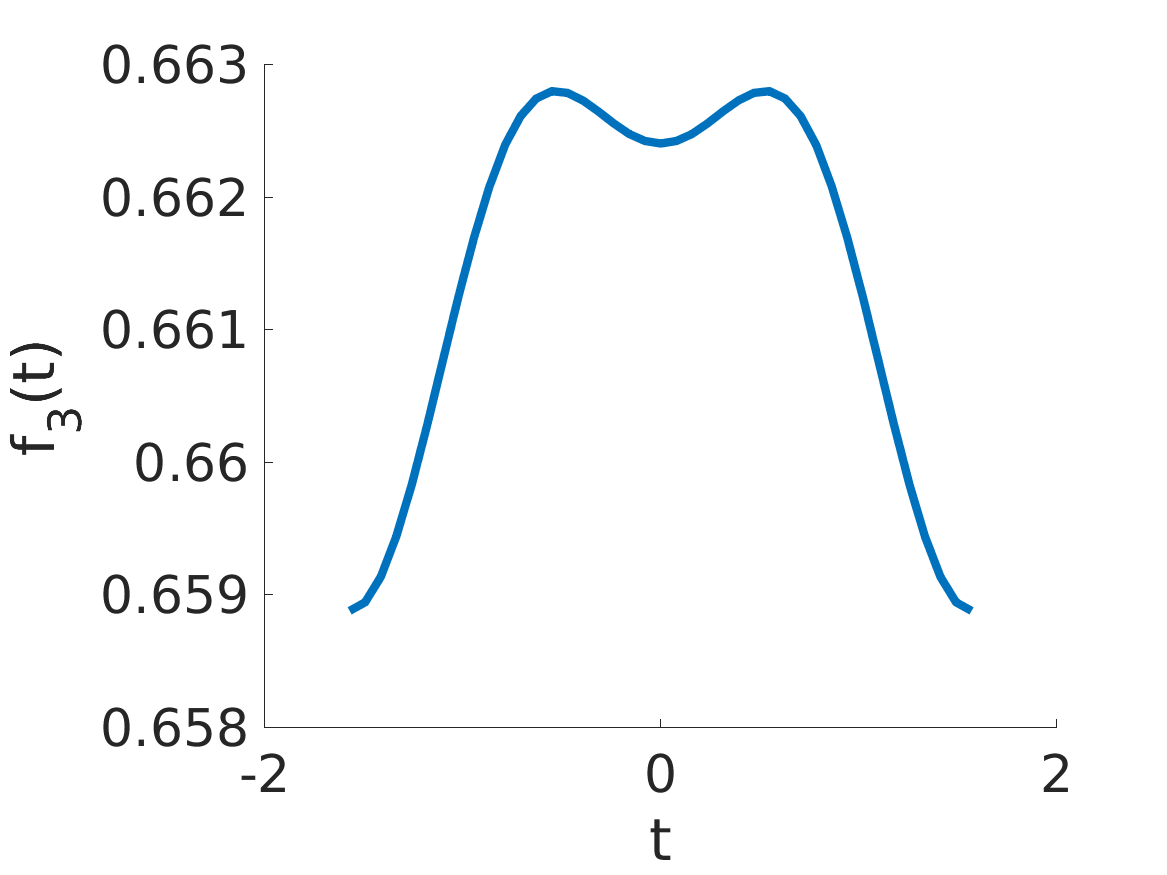}
      \ 
      \includegraphics[width=0.45\textwidth]{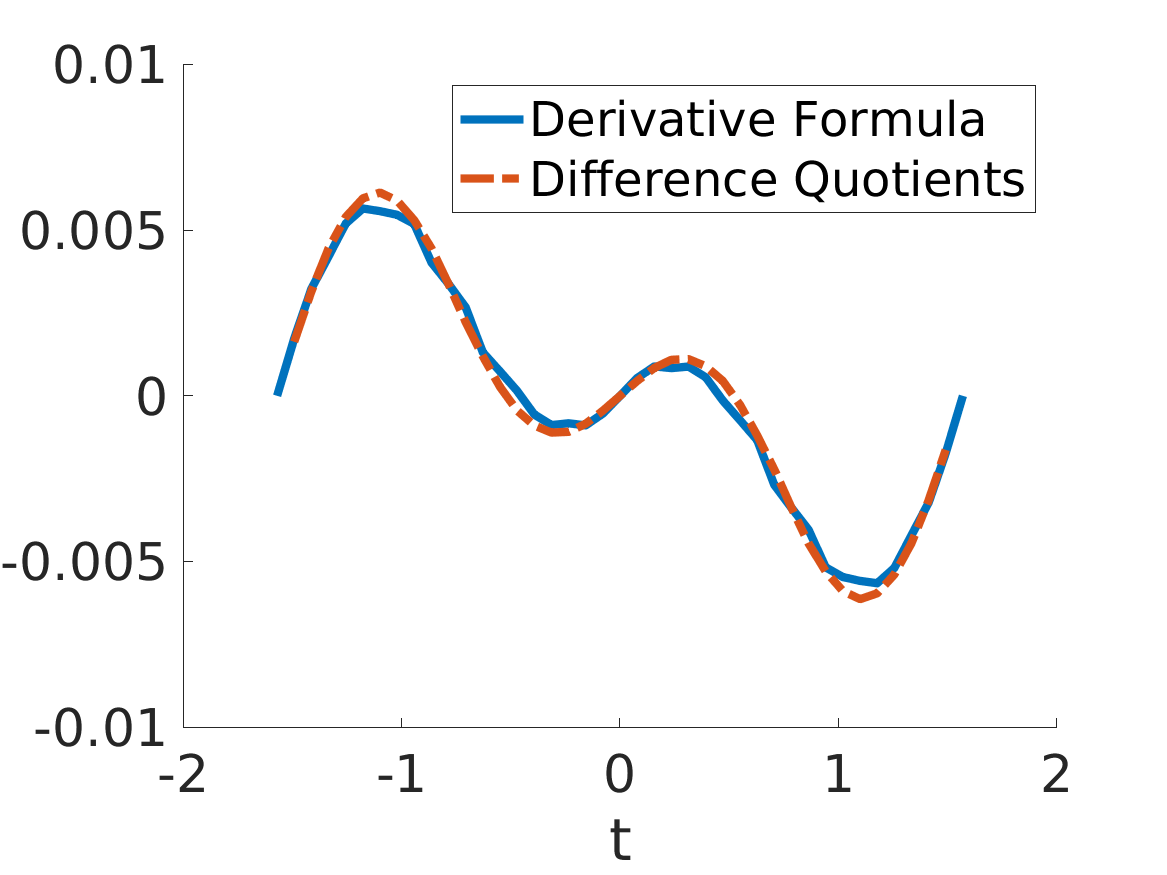}
    \end{center}
    \caption{Left: Plot of $f_3(t)$. Right: Comparison of difference quotients with the derivative formula \eqref{eq:volumeDerivative}.}
    \label{fig:peanuts3}
  \end{figure}
  
  In \cref{fig:peanuts0} we show an example of an optimal membrane shape given the particle configuration $\p$ as obtained from our discretization.
  In \cref{fig:peanuts1}, \cref{fig:peanuts2} and \cref{fig:peanuts3} we evaluate the functions $f_i(t)$ for $-1.4 \leq t \leq 1.4$ and compare the approximation of $f_i'(t)$ by difference quotients with the approximation obtained from evaluating the derivative formula \eqref{eq:volumeDerivative}.
  
  Also in this setting we observe that our formula is generally in good agreement with the approximation by difference quotients.

  \subsection{Gradient flow}
  
  An immediate application of our findings is to employ a gradient flow
  \begin{align*}
    \p'(t) = -\nabla\cJ(\p(t)), \qquad \p(0) = \p_0
  \end{align*}
  in order to investigate stable particle configurations.
  
  In \cref{fig:gradientFlow} we illustrate some time steps
  of the flow for two elliptic particles of different size
  on a square domain $\Omega$.
  Here we assume the boundary conditions
  \begin{align*}
    u|_{\Gamma_i}(y) & = \gamma_{i}
    , \qquad \partial_\nu u|_{\Gamma_i}(y) = 1
  \end{align*}
  for each particle.
  The computations use a discretization of the gradient flow by an explicit Euler scheme
  \begin{align*}
    \p_{k+1}  \colonequals \p_k - \tau \nabla \cJ(\p_k)
  \end{align*}
  with a fixed time step size $\tau > 0$.
  The gradient $\nabla \cJ(\p_k)$ is approximated using the derivative
  formula \eqref{eq:volumeDerivative} for a finite element approximation
  of $u$.
  In fact, the time discrete gradient flow can be viewed as
  gradient descent method with fixed step size for the
  computation of minimizers of $\cJ$.
  Notice, that the simple gradient flow approach was used to simplify
  the presentation and that more sophisticated iterative
  methods based on first order derivatives could be used.
  
  For the given setting with initially unaligned
  particles, the gradient flow leads to a configuration
  where the long axes of the elliptic particles are aligned.
  Furthermore the distance of the particles is initially
  reduced and remains unchanged in a later stage indicating
  that the implicit particle--particle interaction is attractive
  and that this configuration is (close to) a local minimizer
  of $\cJ$.

  \begin{figure}
    \begin{center}
      \includegraphics[width=0.45\textwidth]{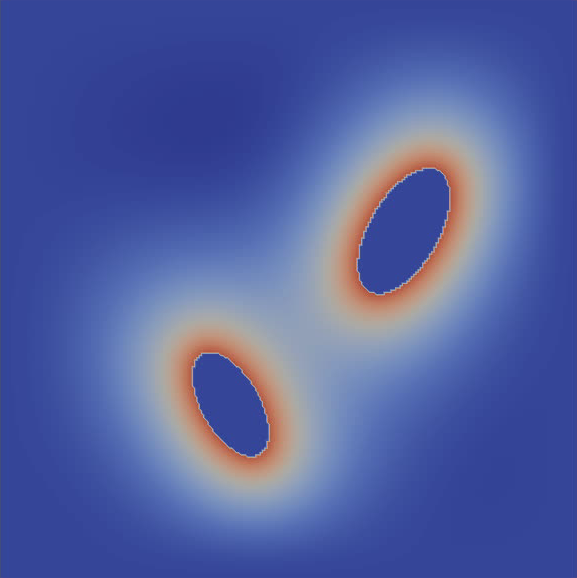}
      \ 
      \includegraphics[width=0.45\textwidth]{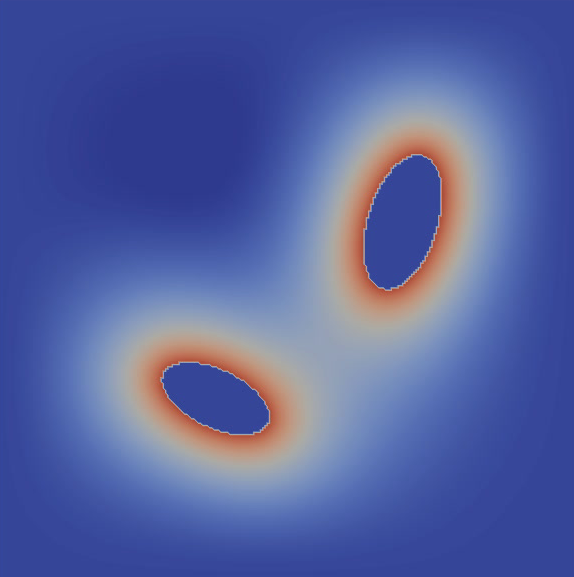}
      \\
      \vspace{0.6em}
      \includegraphics[width=0.45\textwidth]{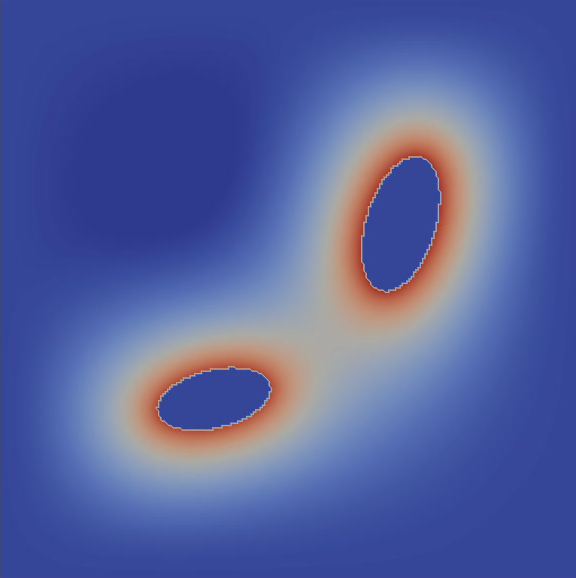}
      \ 
      \includegraphics[width=0.45\textwidth]{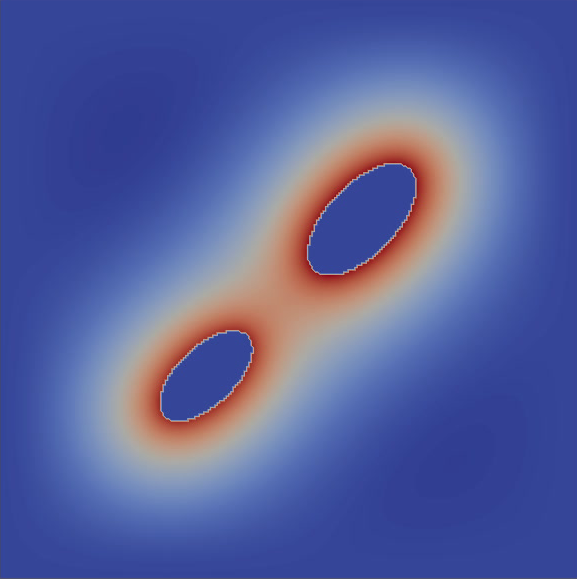}
    \end{center}
    \caption{Time steps $\p_0, \p_5, \p_{10}$, and $\p_{75}$ of a discretized gradient flow for $\cJ$ with two elliptic particles (from left to right, top to bottom).}
    \label{fig:gradientFlow}
  \end{figure}

\section{Conclusion}
  This paper considered a typical model for membrane-mediated particle interactions where the membrane is described as a continuous surface and where the particles are treated as discrete entities that couple to the membrane through certain constraints.
  Based on methods from shape calculus and the implicit function theorem we were able to give a proof for the differentiability of the interaction energy.
  Matrix calculus then enabled us to derive a formula for the first derivative that is numerically feasible in the sense that it can be evaluated from a finite element approximation of the optimal membrane shape for a fixed particle configuration and that it is possible to bound the approximation error of the derivative in terms of the discretization error of the finite element method.
  Numerical examples suggest the correctness of our results.
  
  We emphasize that the approach chosen in this paper is rather general and we expect that it can be used to prove similar results for other model formulations, too.
  Furthermore, as the differentiability proof is based on the implicit function theorem this readily gives constructive instructions on how to derive analogous formulas for higher order derivatives.
  
  Our results allow the efficient differentiation of the interaction potential and may therefore be applied in order to develop new algorithms for investigating stable particle configurations.

\section{Appendix}\label{sec:appendix}

  \begin{theorem}[Whitney extension theorem]\label{thm:whitney}
    Let $A \subseteq \R^n$ be closed, $m \in \N \cup \{\infty\}$, and $f_\alpha \colon A \rightarrow \R$ for all multi-indices $\alpha \in \N^n$ with $\abs{\alpha} \leq m$.
    Suppose that for all multi-indices $\alpha$ with $\abs{\alpha} \leq m$ and all $x,x' \in A$ holds
    \begin{align*}
      f_\alpha(x') = \sum_{\abs{\beta} \leq m - \abs{\alpha}} \frac{f_{\alpha + \beta}(x)}{\beta!} (x'-x)^\beta + R_\alpha(x';x)
    \end{align*}
    where $R_\alpha\colon A \times A \rightarrow \R$ is such that for all $x_0 \in A$ and all $\eps \in \R_{>0}$ there exists $\delta \in \R_{>0}$ such that
    \begin{align*}
      \forall{x,x' \in A\colon} \ \norm{x-x_0} < \delta \ \wedge \ \norm{x'-x_0} < \delta \ \Longrightarrow \abs{R(x';x)} \leq \norm{x-x'}^{m-\abs{\alpha}} \eps\text{.}
    \end{align*}
    Then there exists a function $F \in C^m(\R^n)$ such that $\partial^\alpha F(x) = f_\alpha(x)$ for all $x \in A$ and all multi-indices $\alpha$ with $\abs{\alpha} \leq m$.
  \end{theorem}
  
  \begin{proof}
    See \cite[Theorem I]{Whitney34}.
  \end{proof}

  \begin{theorem}[Implicit function theorem]
    \label{thm:IFT}
    Let $X$, $Y$, $Z$ real Banach spaces, $A \subseteq X \times Y$ open, $F\colon A \rightarrow Z$ and $(x_0,y_0) \in A$ such that $F(x_0,y_0) = 0$.
    Suppose that the partial Fr\'{e}chet-derivative $F_y$ exists on $A$, and $F$ and $F_y$ are continuous in $(x_0,y_0)$.
    
    If $F_y(x_0,y_0)$ is invertible, then there exists an open neighborhood $\cB(x_0)$ and a unique function $y\colon \cB(x_0) \rightarrow Y$ such that $(x,y(x)) \in A$ and $F(x,y(x)) = 0$ for all $x \in \cB(x_0)$.
    Furthermore, if $F \in C^m(A,Z)$ for some $m \in \N$, then also $y \in C^m(\cB,Y)$.
  \end{theorem}
  
  \begin{proof}
    See \cite{HiGr27}.
  \end{proof}

  \begin{lemma}[Transformation of derivatives]
    \label{thm:transformationDerivatives}
    Suppose $X\colon \Omega_1 \rightarrow \Omega_2$ is a diffeomorphism and let $u \in H^2(\Omega_1)$.
    Then
    \begin{align*}
      \int_{\Omega_2} \kappa (\Delta (u \circ X^{-1}))^2 + \sigma \norm{\nabla (u \circ X^{-1})}^2 \d{x}
      & = \int_{\Omega_1} \kappa \frac{\div\left( A \nabla u\right)^2}{\abs{\det DX}} + \sigma \norm{\nabla u}_A^2 \d{x}
    \end{align*}
    where
    \begin{align*}
      A(x) \colonequals \abs{\det DX(x)} \, (DX(x))^{-1}  (DX(x))^{-T}
    \end{align*}
    and
    \begin{align*}
      \norm{\nabla u(x)}_{A(x)}^2 \colonequals \nabla u(x)^T A(x)\nabla u(x)\text{.}
    \end{align*}
  \end{lemma}
  
  \begin{proof}
    First note that for $x \in \Omega_2$ and $v \in H^1(\Omega)$
    \begin{align}
      \label{eq:helper1707071054}
      \nabla (v \circ X^{-1} )(x)
      & = (D(X^{-1})(x))^T \nabla v(X^{-1}(x))
      = (DX(X^{-1}(x)))^{-T} \nabla v(X^{-1}(x))
    \end{align}
    holds almost-everywhere.
    Equation \eqref{eq:helper1707071054} together with the transformation formula applied to the diffeomorphism $X$ we obtain for all $v,w \in H^1(\Omega)$
    \begin{align}
      \label{eq:helper1707071100}
      \begin{aligned}
        & \int_{\Omega_2} \left(\nabla (v\circ X^{-1})(x)\right)^T \nabla (w \circ X^{-1})(x) \d{x}
        \\ & \qquad = \int_{\Omega_2} \nabla v(X^{-1}(x))^T (DX(X^{-1}(x)))^{-1} (DX(X^{-1}(x)))^{-T} \nabla w (X^{-1}(x)) \d{x}
        \\ & \qquad = \int_{\Omega_1} \nabla v(x)^T DX(x)^{-1} DX(x)^{-T} \nabla w(x) \abs{\det DX(x)} \d{x}
        \\ & \qquad = \int_{\Omega_1} \langle \nabla v(x), \nabla w(x) \rangle_{A(x)} \d{x}
        \text{.}
      \end{aligned}
    \end{align}
    By integration by parts and \eqref{eq:helper1707071100} we know that for all $\phi \in C_0^\infty(\Omega_1)$ holds
    \begin{align}
      \label{eq:helper1707071106}
      \begin{aligned}
        \int_{\Omega_2} & \Delta (u \circ X^{-1})(x) \, (\phi \circ X^{-1})(x) \d{x}
        \\ & = -\int_{\Omega_2} \nabla (u \circ X^{-1})(x) \cdot \nabla (\phi \circ X^{-1})(x) \d{x}
          + \int_{\partial\Omega_2} \partial_\nu (u \circ X^{-1})(x) \, (\phi \circ X^{-1})(x) \d{x}
        \\ & = - \int_{\Omega_1} \left( A(x) \nabla u(x) \right) \cdot \nabla \phi(x) \d{x}
        \\ & = \int_{\Omega_1} \div\left( A(x) \nabla u(x) \right) \, \phi(x) \d{x}
          - \int_{\partial\Omega_1} \left( A(x) \nabla u(x) \right) \partial_\nu \phi(x) \d{x}
        \\ & = \int_{\Omega_1} \div\left( A(x) \nabla u(x) \right) \, \phi(x) \d{x}
      \end{aligned}
    \end{align}
    where the boundary terms vanish as of $\phi\circ X^{-1} \in C_0^\infty(\Omega_2)$ and $\phi \in C_0^\infty(\Omega_1)$, respectively.
    On the other hand, application of the transformation formula also yields
    \begin{align}
      \label{eq:helper1707071107}
      \begin{aligned}
        \int_{\Omega_2} \Delta (u \circ X^{-1})(x) \, (\phi \circ X^{-1})(x) \d{x}
        & = \int_{\Omega_1} \Delta (u \circ X^{-1})(X(x)) \, \phi(x) \, \abs{\det DX(x)} \d{x}
        \text{.}
      \end{aligned}
    \end{align}
    Combining \eqref{eq:helper1707071106} and \eqref{eq:helper1707071107} leads to
    \begin{align*}
      \int_{\Omega_1} \abs{\det DX(x)}\Delta (u \circ X^{-1})(X(x)) \, \phi(x) \d{x}
      & = \int_{\Omega_1} \div\left( A(x) \nabla u(x) \right) \, \phi(x) \d{x}
    \end{align*}
    for all $\phi \in C_0^\infty(\Omega_1)$.
    The fundamental theorem of calculus of variations then readily implies that
    \begin{align}
      \label{eq:helper1707071113}
      \Delta (u \circ X^{-1})(X(x)) & = \frac{\div\left( A(x) \nabla u(x) \right)}{\abs{\det DX(x)}}
    \end{align}
    holds for almost-every $x \in \Omega_1$.
    Because $X$ is a diffeomorphism, this expression is well-defined as of $\abs{\det DX(x)} \neq 0$ for all $x \in \Omega_1$.
    Hence, by virtue of the transformation formula and \eqref{eq:helper1707071113} we obtain
    \begin{align}
      \label{eq:helper1707071115}
      \begin{aligned}
        \int_{\Omega_2} (\Delta (u \circ X^{-1})(x))^2 \d{x}
        & = \int_{\Omega_1} \left(\Delta (u \circ X^{-1})(X(x)) \right )^2 \abs{\det DX(x)} \d{x}
        \\ & = \int_{\Omega_1} \frac{\div\left( A(x) \nabla u(x) \right)^2}{\abs{\det DX(x)}} \d{x}\text{.}
      \end{aligned}
    \end{align}
    Finally, the desired assertion is a direct consequence of \eqref{eq:helper1707071100} and \eqref{eq:helper1707071115}.
  \end{proof}

\bibliography{bibliography}
\bibliographystyle{siam}

\end{document}